\documentclass[11pt]{article}
\usepackage[T1]{fontenc}
\usepackage[utf8]{inputenc}
\usepackage[english]{babel}
\usepackage{lmodern}
\usepackage{array}
\usepackage{amsmath}
\usepackage{amsthm}
\usepackage{amssymb}
\usepackage{amsfonts}
\usepackage{fancyhdr}
\usepackage{enumitem}
\usepackage{stmaryrd}
\usepackage{dsfont}
\usepackage{xcolor}
\usepackage{csquotes}
\usepackage{float}
\usepackage{tikz}
\usepackage{tikz-cd}
\usepackage[margin=2cm,nohead]{geometry}
\usepackage{hyperref}
\usepackage[affil-it]{authblk}
\usepackage{orcidlink}

\newcommand{\Q}{\mathbb{Q}}
\newcommand{\R}{\mathbb{R}}

\newcommand{\N}{\mathbb{N}}
\newcommand{\C}{\mathbb{C}}
\newcommand{\h}{\mathcal{H}}
\newcommand{\p}{\mathbb{P}}

\newcommand{\ana}{\mathbf{A}}
\newcommand{\Rat}{\mathrm{Rat}}

\newcommand{\SL}{\mathrm{SL}}
\newcommand{\GL}{\mathrm{GL}}
\newcommand{\res}{\mathrm{res}}
\newcommand{\M}{\mathrm{M}}
\newcommand{\pr}{\mathrm{pr}}

\newcommand{\Spec}{\mathrm{Spec}}

\newcommand{\Int}{\mathrm{Int}}

\newtheorem{thm}{Theorem}[section]
\newtheorem*{theoreme}{Theorem}
\newtheorem{theorem}{Theorem}

\newtheorem{prop}[thm]{Proposition}
\newtheorem{proposition}{Proposition}

\newtheorem{nota}[thm]{Notation}
\newtheorem{lem}[thm]{Lemma}

\newtheorem{defi}[thm]{Definition}

\newtheorem{rem}[thm]{Remark}

\newtheorem{cor}[thm]{Corollary}

\newtheorem{ex}[thm]{Example}

\title{Continuity of the Julia sets in non-archimedean and hybrid settings}
\author{Alexandre Roy \orcidlink{0009-0004-0193-008X}%
  \thanks{Electronic address: \texttt{alexandre.roy@unicaen.fr}}}
\affil{Université Caen Normandie, CNRS, Normandie Univ, LMNO UMR6139, F-14000 CAEN, FRANCE }
\date{}

\begin{document}

\maketitle
\begin{abstract}
Let $k$ be a complete, non-archimedean. We study the dynamics of families of one-variable rational functions parametrized by Berkovich spaces over $k$ and over $\C$ equipped with the hybrid norm. In both settings, we show an analogue of $\mathrm{Ma\tilde{n}\acute e -Sad-Sullivan}$'s theorem about the continuity of the Julia set (in the Hausdorff topology). The non-archimedean case relies on the maximum modulus principle, which we show for analytic functions over analytic spaces without boundary.
\end{abstract}

\section{Introduction}

Y. Luo in \cite{LuoTreesBarycentricExtension} studied degenerations of sequences of rational maps and constructed a limit rational map defined over a non-archimedean field. And he uses the description of the Julia set of this limit rational map to obtain a result over the complex rational map. 
Namely, he showed the following result:
\begin{theoreme}$\mathrm{(Luo, Theorem 1.5)}$

A hyperbolic component $\mathcal{H}$ admits bounded escape if and only if it has nested Julia sets.
\end{theoreme}

Admitting bounded escape means that there exists $[f_n] \in \mathcal{H}$ which degenerates and such that all of its multipliers remain bounded. 

So he manages to control the Julia set of complex rational maps using the Julia set of the constructed non-archimedean limit. In his case, the Julia set of the limit rational map is entirely contained in the hyperbolic space $H_k \subset \p^{1,an}_k$ of the Berkovich projective line. In this situation the Julia set of the sequence $f_n$ will not degenerate continuously towards the Julia set of the limit rational map. Indeed to hope for such continuity, one needs that there are type I points in the Julia set of the limit rational map. 

The aim of this article is to obtain some conditions ensuring that the Julia set of the complex rational maps degnerates continuously towards the Julia set of the non-archimedean limit map. In order to have a formal definition of this continuity, one needs to work over some compactification of the space, $\Rat_d$, of rational maps of degree $d$. 

Studying degeneracies of archimedean objects using non-archimedean limits as did Luo comes from numerous work over time. For example, J. Kiwi \cite{KiwiPolynomialDynamics} studied the degeneracy of merophormic families of rational maps using the field of Laurent series. Similar ideas were also used by C. Favre \cite{FavreDegenerationInHybridSpace}, \cite{FavreBlowupMultipliers} to work over meromorphic families of rational maps. At the same time, L. DeMarco and X. Faber \cite{DeMarcoFaber} used this kind of techniques to study a degenerating limits of equilibrium measures.

At the same time, L. DeMarco and T. McMullen \cite{DeMarcoMcMullenTreesDynamicsPolynomials} studied a compactification of the space of polynomial maps and they get a boundary which is a $\R$-tree and seems to be connected with the hybrid spaces as introduced by V. Berkovich.

Studying this hybrid spaces have given new results in the study of these degenerating phenomena and we can, for example, quote the work of J. Poineau \cite{PoineauDynamique} who showed that over this hybrid spaces (and over a more general class of spaces) the equilibrium measure is continuous.

Simultaneously, J. Poineau defined in \cite{PoineauCompactificationHybride} (see subsection \ref{constructionHybridCompactification}) the valuative compactification of any variety $X$ over a non-trivially valued field $k$. In this compactification, denoted $X^\urcorner$, the analytification, $X^{an}$, of $X$ over $k$ (which coincides with $X(\C)$ if $k= \C$) is embedded as an open subset and the boundary is a quotient of an open subset of the analytification of $X$ over $k$ trivially valued. So, in the case where $k= \C$, the boundary is an object of non-archimedean nature and so it gives the same idea as the limit rational map constructed by Luo. The fact that the Berkovich spaces can give a different description of Luo's construction was also proved by C. Favre and C. Gong in \cite{FavreGong}. Their description gave a numerous number of results, see for example \cite{GongMultiplierScale}, \cite{FavreGongRescalingLimits}, \cite{LooperYap}.  

Here, we will work over $\Rat_d^\urcorner$ which is the valuative compactification of $\Rat_d$ defined by Poineau. Since the boundary of this space is locally a Berkovich space over $\C$ trivially valued, the first question we would like to investigate is the continuity of the Julia set in the non-archimedean case. 

\bigskip
\begin{center}
\textbf{The non-archimedean case}
\end{center}

In a famous paper \cite{MSS}, R. Ma$\mathrm{\tilde{n}}$é, P. Sad and D. Sullivan studied the stability of the Julia set of a family of rational maps parametrized by a complex analytic spaces. Their result can be stated as follows.

\begin{theoreme} $\mathrm{(Ma\tilde{n}\acute e -Sad-Sullivan)}$

Let $X$ be a complex connected manifold and $f : X\times \p^1(\C) \rightarrow X\times \p^1(\C)$ be a holomorphic family of rational maps parametrized by $X$. For $\lambda \in X$, let us denote $f_\lambda : \p^1(\C) \rightarrow \p^1(\C)$ the associated rational map.

Let $x\in X$. Then the following conditions are equivalent :
\begin{enumerate}
\item The number of attracting cycles of $f_\lambda$ is locally constant at $x$.
\item The maximum period of an attracting cycle of $f_\lambda$ is locally bounded at $x$. 
\item The Julia set moves holomorphically at $x$. 
\item For all $y$ sufficiently close to $x$, every periodic point of $f_y$ is attracting, repelling or persistently indifferent.
\item The Julia set depends continuously on $\lambda$ (in the analytic topology) on a neighborhood of $x$.
\end{enumerate}
\end{theoreme}

A goal of this article is to get a counterpart of this theorem in a different setting. First, one wants to replace $\C$ by a non-archimedean complete valued field, $k$. T. Silverman studied this question in \cite{TSilvermanNonArchiLambdaLemma}. He wanted to show an equivalence which looked like $3 \iff 4$ by replacing a holomorphic motion by an analytic motion. He manages to get an equivalence in the case where $k$ is algebraically closed and $X$ is of dimension one. 

Unfortunately, condtion $3$ is not equivalent to the others. It is because in the analogous of the point $4$ that he gets, there is an extra condition on the repelling periodic point. In Silverman's work, the indifferent point must not be unstable which as we will see in lemma \ref{Unstably and not persistently indifferent is the same} is the exact counterpart of not persistently indifferent but the repelling points must also have ,,multiplicity one''. This means that if $f$ is a rational map over a non-archimedean field $k$ then, all its repelling points must be in $k$ and not in its algebraic closure. The problem is that if we take $X$ a $k$-analytic space and $x\in X\backslash X(k)$, then in general the residue field at $x$ will be complete but not algebraically closed. Hence the repelling point of the rational map associated to $x$ will likely have repelling points not in its residue field.

Silverman showed in his example 4.6 that the Julia set of the family $f_\lambda(z) = z^2 + \lambda$ for $\lambda\in \ana^{1,an}_k \backslash D^{an}(0,1)$ does not move holomorphically but the Julia set moves continuously (see example \ref{z^2+ T has only repelling points}). 

The first goal of this article is to focus on point $5$ in this equivalence and get a similar theorem over any complete valued field (even trivially valued) and with $X$ of any dimension. 
\begin{theorem}(infra theorem \ref{JuliaSetMovesContinuouslyNonArchi})

Let $k$ be a non-archimedean complete valued field and $X$ a $k$-analytic space with no boundary. Let $f$ be an analytic family of rational maps and $x \in X$. Let us assume that $f_x$ has type I repelling points. Then the following are equivalent :
\begin{itemize}
    \item The Julia set depends continously on $\lambda$ on a neighborhood of $x$, 
    \item For all $y$ sufficiently close to $x$, every point of $f_y$ is attracting, repelling or persistently indifferent. 
\end{itemize}
\end{theorem}

The only difference with the theorem of Ma$\mathrm{\tilde{n}}$é-Sad-Sullivan is that we require the rational map to have type I repelling points which is not a condition that exists in the complex setting since it is always true. And we will see in the section \ref{Section continuity in the hybrid settings} that this condition is necessary to have a continuity of some degenerating Julia sets.

Furthermore, if the residual characteristic of $k$ is $0$ or bigger than the degree of the family of rational maps, we can also show the points $1,2$ and $5$ are equivalent.

\begin{proposition}(infra propositon \ref{JuliaSetsMovesContinuouslyIfAttractingBounded})

Let $k$ be a non-archimedean complete valued field and $X$ a $k$-analytic space with no boundary. Let $f$ be an analytic family of rational maps of degree $d \geq 2$. Let us assume that $\mathrm{char}(\tilde{k}) = 0$ or $\mathrm{char}(\tilde{k}) = p >d$. Let $x\in X$ such that $f_x$ has type I repelling points. 

Then the following are equivalent :
\begin{enumerate}
    \item The Julia set depends continously on $\lambda$ on a neighborhood of $x$, 
    \item The number of attracting cycles is constant on a neighborhood of $x$,
    \item The maximum period of an attracting cycle of $f_\lambda$ is locally bounded at $x$.
\end{enumerate}
\end{proposition}

All of these results come from the fact that the maximum modulus principle, which states that a holomorphic map atteins its maximum if and only if it is constant, is also true over a non-archimedean field, in the sense that if an analytic map atteins its maximum then, its absolute value is constant (see theorem \ref{maximum modulus}). This result has already been proven, independently, in a more general setting by W. Gubler-J. Rabinoff in \cite{GublerRabinoffPSHFunctions}, Theorem 1.5. The two proofs are using different tools. Our proof relies directly on the theory of graded reduction of M. Temkin \cite{TemkinReductionGradueeII}. Their proof comes from connectedness by curves of analytic spaces and then relies on results on psh- functions on curves which by construction relies on graph theory.

As a consequence of these theorem, we show that the continuity of the Julia set is independant of the field. In the sense that if the Julia set moves continuously on a $k$-analytic space $X$ and for any field extension $K$ of $k$, the Julia set moves continously on $X_K = X \hat{\otimes} K$. Furthermore, we show that if $x\in X$ then the Julia set moves continuously on a neighborhood of $x$ if and only if it moves continuously on a neighborhood of all $y \in (\pi^K_X)^{-1}(x)$. We refer the reader to subsection \ref{subsection continuity independant of the field} for the precise statements.

Lee \cite{LeeJStability} and Benedetto-Lee \cite{BenedettoLeeJStability} have studied some phenomena of stability of the Julia set. They showed that the expanding rational maps over a non-archimedean field are J-stable. The tools they use are really different of ours since they only consider rational maps of $\Rat_d^{an}(k)$ and not all the points of the Berkovich analytification of $\Rat_d$.

Finally, we also show that if the Julia set moves continuously on a neighborhood of $f\in \Rat_d^{an}$, then for any $M\in \SL^{an}_{2,\mathcal{H}(f)}$ the Julia set moves also continuously on a neighborhood of $M\cdot f$. Since the set $\SL^{an}_{2,\mathcal{H}(f)}$ which acts on $f$ depends on the point, it is different than in the complex case. This allows us to have some results about the stability of the Julia set on $\M_d^{an}$ and not only on $\Rat_d^{an}$ (see subsection \ref{subsection continuity on M_d^an}). 

\bigskip
\begin{center}
\textbf{The degenerating case}
\end{center}

Now, we would like to look at the continuity of the Julia set in the case of some degenerating family of rational maps. So, we would like to look at the stability of the Julia set over $\Rat_d^\urcorner$. Since, it is a quotient space of an open subset of the hybrid analytification of $\Rat_d$, we will study the stability of the Julia in $\Rat_d^{hyb}$. Let us recall the idea of this analytification. The hybrid analytification of a variety $X$ over a non-trivially valued field $k$ is given with a projection $\pr : X^{hyb} \rightarrow [0,1]$. For each $\varepsilon >0$, $\pr^{-1}(\varepsilon)$ is the analytification of $X$ over the completion of $k$ equipped with its absolute value to the power $\varepsilon$ and $\pr^{-1}(0)$ is the analytiifaction of $X$ over $k$ equipped with the trivial absolute value.

The first idea would be that the Julia set will be stable over some open subset $U \subset \Rat_d^{hyb}$ if it is stable on $U_0 = U \cap \Rat_d^{an,0}$ (which corresponds to the boundary of the compactification) and on $U_{>0} = U \cap \Rat_d(\C)$. Unfortunately it is not true, see example \ref{continuity in hybrid is not continuity on U_0 and U_>0}. The fact that the Julia set degenerates continuously is a strong condition and it is in general not true. But there is still a similar result as the theorem of Ma$\mathrm{\tilde{n}}$é-Sad-Sullivan in this setting. We have a equivalence which looks like $4 \iff 5$.

\begin{theorem}(infra theorem \ref{Julia set moves continuously hybrid setting})

Let $\C$ be equipped with the hybrid norm and $X$ a flowing $\C_{hyb}$-analytic space. Let $f$ be an analytic family of rational maps and $x \in X$. 

\begin{itemize}
\item If $x \in X_{>0}$, then the Julia set moves continuously on a flowing neighborhood of $x$ if and only if it moves continuously on a neighborhood $U_{\pr(x)}$ of $x$ seen in the fiber $X_{\pr(x)}$,

\item If $x\in X_0$, then the Julia set moves continuously on a flowing neighborhood of $x$ if and only if there is a flowing neighborhood $U$ of $f_x$ such that :
\begin{enumerate}
\item the Julia sets moves continuously on $U_{>0}$,
\item for all $y\in U_0$, every indifferent point of $f_y$ is hybridly persistently non-repelling.
\end{enumerate} 
\end{itemize} 
\end{theorem}
The main difference here is that one requires that the indifferent point are ,,hybridly persistently non-repelling''. This means that they either come from some complex attracting point or some complex indifferent Fatou point. In this setting, there is no maximum modulus principle and so, the non-archimedean indifferent point could come from complex attracting points. We also study the continuous motion over the hybrid compactification of $\M_d$ and show that it is equivalent for the Julia set to moves continuously on a neighborhood of $f\in \Rat_d^\urcorner$ or of $[f] \in \M_d^\urcorner$.

We conclude by making a connection between the fact that the Julia set of a sequence of rational maps degenerates continuously and the fact that the multipliers of the sequence of rational maps blow-up in the sense of \cite{FavreBlowupMultipliers}, \cite{GongMultiplierScale}. And we show that if a sequence of rational maps is in some good hyperbolic connected component, then its Julia set degenerates continuously if and only if all its multipliers blow-up.

\begin{theorem}(infra theorem \ref{InGoodHyperbolicComponentsJuliaSetContinuousIFFMulitpliersBlowUp})

Let $U$ be a hyperbolic connected component of $\M_d(\C)$ such that the Julia sets do not have indifferent periodic point. Let $\tilde{f}_i$ be a degenerating net in $U$.

Then the Julia set of $\tilde{f}_i$ degenerates continuously if and only if all of the multipliers of its repelling points blow-up.
\end{theorem}

\begin{center}
\textbf{Organisation of the article}
\end{center}

In section \ref{sectionEspacesBerko}, we recall the definition of the Berkovich spaces over a complete field and over a hybrid field. This allows us to recall the definition of the valuative compactification defined by J. Poineau in \cite{PoineauCompactificationHybride} which will be the focus on section \ref{Section continuity in the hybrid settings}. Finally, we prove a maximum modulus principle in the case of a Berkovich space over a complete non-archimedean field, see theorem \ref{maximum modulus}. In section \ref{sectionJuliaSet}, we recall the definition of the Julia set over a non-archimedean field and we recall the known results about the periodic points in this case. We show that the Julia set can be completely understood by taking any field extension of the base field, see corollary \ref{Taking the Julia set over any extension does not change it}.  In section \ref{Section continuity in the non archimedean settings}, we show the analogue of the Ma$\mathrm{\tilde{n}}$é-Sad-Sullivan theorem in the non-archimedean case and give some examples. Then, we show that the continuity does not dependant on the field. Finally, we study the continuous motion on $\M_d^{an}$. In the section \ref{Section continuity in the hybrid settings}, we study first the polynomial case, then we show the equivalent of the Ma$\mathrm{\tilde{n}}$é-Sad-Sullivan theorem in the hybrid case and study the continuous motion on the compactification, $\M_d^\urcorner$ of $\M_d(\C)$. We conclude by linking the continuous degeneracy of the Julia set with the blow-up of multipliers.  

\begin{center}
\textbf{Acknowledgement}
\end{center}

I would like to warmly thank Jérôme Poineau for his many advices, many ideas, his support throughout this work and his careful reading. I am also thankful to Thomas Gauthier for his time, our discussion about this subject and for pointing out the result of T. Silverman on this matter and to Charles Favre for answering my questions.
  
\section{Berkovich spaces}\label{sectionEspacesBerko}

In this section, we first give the definition of the Berkovich spaces over a complete valued field which will be the setting of the section \ref{Section continuity in the non archimedean settings} to study the continuity of the Julia set. Then, we define the hybrid spaces which will be the setting of the section \ref{Section continuity in the hybrid settings}. We also define the hybrid compactification as defined by J. Poineau in \cite{PoineauCompactificationHybride}, since we will apply some results to the hybrid compactification$, \Rat_d^\urcorner$, of $\Rat_d$. Finally, we will present the graded reduction in the sense of M. Temkin \cite{TemkinReductionGradueeII} and use it to get a maximum modulus principle in the case of Berkovich spaces over a complete valued field.
 
\subsection{Berkovich spaces over Banach rings}

Let $(A, \vert\vert \cdot \vert\vert)$ be a Banach ring. For this first subsection, we use the definitions of V. Berkovich \cite{BerkovichLivre}.

Let us start with the definition of the analytification of $\ana^n_A$, denoted by $\ana^{n,an}_A$, which is the affine space of dimension $n$ over $A$. Let us focus on the underlying topological space although this space has a structure of a locally ringed space.

\begin{defi} 
Let us denote $\ana^{n,an}$ the affine space of dimension $n$ over $A$.

The underlying space is the set of all bounded multiplicative semi-norms over $A[T_1, \cdots, T_n]$. Hence, it is the set of applications :
\begin{equation*}
|\cdot| : A[T_1, \cdots, T_n] \rightarrow \R_{\geq 0}
\end{equation*}
such that :
\begin{itemize}
\item $|0| = 0$ and $|1| = 1$,
\item $\forall P,Q \in A[T_1, \cdots, T_n], |PQ| = |P| |Q|$,
\item $\forall P,Q \in A[T_1, \cdots, T_n], |P + Q| \leq |P| + |Q|$,
\item $\forall a\in A, |a| \leq \vert\vert a \vert\vert.$
\end{itemize}

We denote by $\mathcal{M}(A)$ the space $\ana^{0,an}_A$ and call it the Berkovich spectrum of $A$. We have a projection $\pr : \ana^{n,an}_A \rightarrow \mathcal{M}(A)$ induced by the injection $A \hookrightarrow A[T_1, \cdots, T_n]$.

If $x \in \ana^{n,an}_A$, we denote $|\cdot|_x$ the associated semi-norm. The ring $A[T_1, \cdots, T_n]/(ker~|\cdot|_x)$ is an integral domain, so we may look at its fraction field. As $|\cdot|_x$ is inducing an absolute value on it, we may take its completion and we denote it $\mathcal{H}(x).$

There is a topology on $\ana^{n,an}_A$. Let us take the coarsest topology such that for all $P \in A[T_1, \cdots, T_n]$, the applications
\begin{center}
$\begin{cases}
\ana^{n,an}_A &\rightarrow \R_{\geq 0},\\
|\cdot|_x &\mapsto |P|_x
\end{cases}
$
\end{center}
are continuous. With this topology, $\ana^{n,an}_A$ is Hausdorff and locally compact. Furthermore, $\mathcal{M}(A)$ is compact. The projection $\pr : \ana^{n,an}_A \rightarrow \mathcal{M}(A)$ is continuous. 
\end{defi}

\begin{ex}
An example of a Banach ring we will use is the hybrid field. Let $(k,|\cdot|)$ be non-trivially valued field, then we define on $k$ a hybrid norm $|\cdot|_{hyb}$ such that
\begin{center}
$|x|_{hyb} = \max(|x|, |x|_{trivial})$.
\end{center}
It gives a Banach ring. 
\end{ex}

\begin{prop}
If $X$ is a scheme of locally finite presentation over $A$ where $A$ is a geometric base ring, which includes valued and hybrid fields, then we may consider its analytification. It is a $A$-analytic space in the sense of Berkovich which we denote by $X^{an}$.
\end{prop}

\begin{rem}
For a precise definition of geometric base ring, we may refer for example to the book of T. Lemanissier - J. Poineau (\cite{LemanissierPoineau}, Définition 3.3.8).
\end{rem}

Let us recall how to define this analytification based on the proof of Lemanissier-Poineau, Théorème 4.1.4 \cite{LemanissierPoineau} and one can find this analytification in the work of Berkovich.

\begin{itemize}
\item First step : If $X = \ana^n_A$, then $X^{an} = \ana^{n,an}_A$.

\item Second step : If $X$ is a closed subscheme of $\ana^n_A$, then it is defined by an ideal $I$ finitely generated of $\mathcal{O}(\ana^n_A)$ and $X^{an}$ is the closed analytic of $\ana^{n,an}_A$ defined by the sheaf of ideals generated by $I$. 

\item Last step : If $X$ is locally of finite presentation, then $X = \bigcup U_i$ where the $U_i$ are affine varieties locally of finite presentation which are analytifing as previously, then $X^{an}$ is obtained by gluing the $U_i^{an}$.
\end{itemize}

\begin{prop}
If $X$ and $Y$ are two $A$-schemes locally of finite presentation and $f : X \rightarrow Y$ a scheme morphism, then we can consider its analytification $f$ to get $f^{an} : X^{an} \rightarrow Y^{an}$.
\end{prop}

This analytification preserves most of the properties of the morphism of schemes. In the case where the morphism is finite, we have the following result : 

%\begin{lem}Proposition 5.6.4 de \cite{LemanissierPoineau} \label{propriétés des morphismes passent aux changements de base}

%Soit $\mathcal{A}$ un anneau de base géométrique. Soit $\phi : X \rightarrow Y$ un morphisme d'espace $\mathcal{A}$-analytique. Soit $x \in X$ et supposons que $\phi$ est fini et plat en $x$. 

%Soit $\psi : Z \rightarrow Y$ un morphisme d'espaces $\mathcal{A}$-analytiques, alors le morphisme $\phi_Z : X \times_Y Z \rightarrow Z$ est fini et plat en tout point au-dessus de $x$. 
%\end{lem}

\begin{lem}\label{morphisme fini sur meme dimension est ouvert} Lemma 3.2.4 of \cite{BerkovichLivre}

Let $\phi : X \rightarrow Y$ a finite morphsim of $k$-analytic spaces such that $\dim(X) = \dim(Y)$ and $X$ is locally irreducible. Then $\phi$ is an open morphism.
\end{lem}

\begin{rem}
The fact that the morphsim has finite fibers is enough see for example Théorème 5.2.9 of \cite{LemanissierPoineau}. 
\end{rem}

%T. Lemanissier et J. Poineau ont aussi montré que l'analytification de certains morphismes finis est ouverte.

%\begin{prop}\label{les analytifies des morphismes finis entre schemas reduits sont ouverts} Proposition 5.6.2 \cite{LemanissierPoineau}
%Soit A un anneau géométrique et $\phi : X \rightarrow Y$ un morphisme d'espaces $A$-analytiques. Soit $x \in X$ tel que $\phi$ est fini en $x$. Supposons que $\mathcal{O}_{Y, \phi(x)}$ est réduit et que $\mathcal{O}_{X,x}$ est un $\mathcal{O}_{Y, \phi(x)}$-module sans torsion. Alors le morphisme $\phi$ est ouvert en $x$.

%En particulier si $X, Y$ sont deux $A$-schéma réduits, localement de présentation finie et si $\phi : X \rightarrow Y$ est un morphisme de schéma fini. Alors, $\phi^{an} : X^{an} \rightarrow Y^{an}$ est ouvert.  
%\end{prop}

\subsection{Hybrid Berkovich spaces}

Let us now focus on the case where the Banach ring $A$ is a field equipped with a hybrid norm. We present the section 2 of Poineau's article \cite{PoineauCompactificationHybride}, we won't prove the result but recall the definitions.

\begin{prop}
Let $(k, |\cdot|)$ be a valued field and let us denote $k_{hyb}$ the field equipped with the hybrid  norm. Then, the Berkovich spectrum is :
\begin{center}
$\mathcal{M}(k_{hyb}) = [0,1]$,
\end{center} 
where the identification comes from the association to any $0\leq \varepsilon \leq 1$ to the absolute value $|\cdot|^\varepsilon$, $|\cdot|^0$ is the trivial absolute value.

Hence, the residue fields $\mathcal{H}(\varepsilon)$ are the completion of $k$ equipped with the absolue value $|\cdot|^\varepsilon$. Let us denote them $\hat{k}_\varepsilon$.   

A $k_{hyb}$-analytic space $X$ is equipped with a projection $\pr : X \rightarrow \mathcal{M}(k_{hyb})$. 

For all $\varepsilon \in [0,1], \pr^{-1}(\varepsilon) =: X_\varepsilon$ is a $\mathcal{H}(\varepsilon)$-analytic space. 
\end{prop} 

\begin{rem}
If $X$ is a $k_{hyb}$-analytic space, where $k$ is a complete field, for all $\varepsilon > 0, X_\varepsilon$ have isomorphic underlying topologic spaces: all the spaces $\mathcal{H}(\varepsilon)$-analytic spaces are $k$-analytics spaces equipped with a different normalisation. And for $\varepsilon = 0, X_0$ is an analytic space over a trivially valued field.

In the case, where $k$ is an archimedean field, then for $\epsilon > 0$, one has complex analytic spaces and for $\varepsilon = 0$ one finds a Berkovich space of non-archimedean nature. 

Hence, the hybrid fields can link archimedean and non-archimedean phenomena.   
\end{rem}

Poineau introduced the notion of flow which allows to define an equivalence relation mandatory to the construction of a hybrid compactification. Let us present its definition.

\begin{defi}
Let $\varepsilon \in [0,1]$, then we define :
\begin{center}
$I_\varepsilon :=
		\begin{cases}
        [0, +\infty[ & \mathrm{if}~ \varepsilon = 0\\
        [0, \frac{1}{\varepsilon}] & \mathrm{else.}    
        \end{cases}$
\end{center}
Let us denote $I^*_\varepsilon := I_\varepsilon \backslash \lbrace 0 \rbrace$.

Furthermore, if $S$ is a $k_{hyb}$-analytic space, then with the projection $\pr : S \rightarrow \mathcal{M}(k_{hyb})$, for all $x\in S$, we define $I_x := I_{\pr(x)}$.
\end{defi}

\begin{lem}
Let $x \in \ana^{n,an}_{k_{hyb}}$ and $\alpha \in I_x$, then the application :
\begin{center}
$P \in k[T_1, \cdots, T_n] \mapsto |P(x)|^\alpha \in \R_{\geq 0}$ 
\end{center}
defines a point of $\ana^{n,an}_{k_{hyb}}$ which we denote by $x^\alpha$. We have $\pr(x^\alpha) = \alpha ~ \pr(x)$.

Furthermore, if $\alpha \in I^*_x$, then the fields $\mathcal{H}(x)$ and $\mathcal{H}(x^\alpha)$ are isomorphic.
\end{lem}

We may now define the flow:

\begin{defi}
Let 
\begin{center}
$D(\ana^{n,an}_{k_{hyb}}) := \bigcup_{x \in \ana^{n,an}_{k_{hyb}}} \lbrace x\rbrace \times I_x^\alpha \subset \ana^{n,an}_{k_{hyb}} \times \R_{> 0}$.
\end{center}
Then the flow is the application :
\begin{center}
$\Phi : \begin{cases} 
		D(\ana^{n,an}_{k_{hyb}}) &\rightarrow \ana^{n,an}_{k_{hyb}},\\
		(x,\alpha)&\mapsto x^\alpha.
		\end{cases}$
\end{center}
\end{defi}

\begin{prop}
The flow is a continuous and open application.
\end{prop}
For the proof, we may refer to the proposition 2.10 of Poineau \cite{PoineauCompactificationHybride}.

We can now define the notion of a point's trajectory and a set's trajectory.

\begin{defi}
Let $x \in \ana^{n,an}_{k_{hyb}}$, then the trajectory of the point $x$ is the set $T(x)$ defined by :
\begin{center}
$T(x) := \Phi(x, I_x^*) = \lbrace x^\alpha, \alpha \in I_x^*\rbrace.$ 
\end{center}
\end{defi}

\begin{rem}
Let $x, y\in \ana^{n,an}_{k_{hyb}}$. If $y \in T(x)$ then $T(y) = T(x)$. 
\end{rem}

This result will allow us to define an equivalence relation by using the points' trajectories.

We can further define the set's trajectory. 
\begin{defi}\label{def trajectoire ensemble}
Let $V$ be a subset of $\ana^{n,an}_{k_{hyb}}$, then the trajectory of $V$ is the set
\begin{center}
$T(V) := \bigcup_{x\in V} T(x) \subset \ana^{n,an}_{k_{hyb}}$.
\end{center}

If $V = T(V)$, then $V$ will be said to be flowing. 

Furthermore, if $V, V'$ are two subsets of $\ana^{n,an}_{k_{hyb}}$, then $T(V\cup V') = T(V) \cup T(V')$ and $T(V\cap V') = T(V) \cap T(V')$.
\end{defi}

\begin{defi}\label{flowing k hyb -analytic space}
A flowing local $k_{hyb}$-analytic model is the data of an integer $n$, a flowing open $U \subset \ana^{n,an}_{k_{hyb}}$ and a closed analytic subset of $U$.

A flowing $k_{hyb}$-analytic space is a $k_{hyb}$-analytic space obtained by gluing flowing local $k_{hyb}$-analytic models along flowing open subsets.
\end{defi}

\subsection{Construction of the hybrid compactification}\label{constructionHybridCompactification}
In this part, we present the construction of a hybrid compactification, we are following the sections 3 and 4 of Poineau's article \cite{PoineauCompactificationHybride}. We present the main results and we often refer to the article to see the proofs.

First, the construction only works on an open of $X^{hyb}$ where $X$ is a variety on $k$ and $X^{hyb}$ means the analytification $X$ on $k_{hyb}$. The goal is here to remove a generic fiber from $X^{hyb}$.

M. Raynaud (\cite{RaynaudGeometrieRigide}), P. Berthelot (\cite{BerthelotLivre}) , V. Berkovich (\cite{BerkovichSchemaFormel1}, \cite{BerkovichSchemaFormel2}) then A. Thuillier (\cite{ThuillierGeometrieToroidale}) noticed that the non-archimedean spaces can be used to define a notion of a generic fiber on formal schemes. As in the section 3 of Poineau \cite{PoineauCompactificationHybride}, we present Thuillier's construction.

Let $\mathcal{X}$ be a formal scheme on $k_0$, which is the notation for $k$ trivially valued, which is locally algebraic. To this formal scheme, we associate a generic fiber $\mathcal{X}^\beth$ which is a $k_0$-analytic space and an application $r_{\mathcal{X}} : \mathcal{X}^\beth \rightarrow \mathcal{X}_s$ which is anticontinuous i.e. the inverse image of an open is closed.

We present the construction in the affine case, but it exists in a more general setting.

Let $\mathcal{X} = X$ be an affine variety, $X = \Spec A$. Then,
\begin{center}
$X^\beth = \mathcal{M}(A)$
\end{center}
where $A$ is trivially valued. The application $r_X : \mathcal{M}(A) \rightarrow \Spec A$ is the reduction such that for $x\in \mathcal{M}(A),$
\begin{center}
$r_X(x) = \lbrace a\in A, |a(x)| <1 \rbrace$.
\end{center}

This allows us to define the part which will be the the boundary of the hybrid compactification.

\begin{defi}
Let $X$ be a $k$-variety, then let us define
\begin{center}
$X_\infty := X^{an}_0 \backslash X^\beth$
\end{center}
it is an open of $X^{an}_0$. Hence, it is a $k_0$-analytic space.
\end{defi}

Let us look at some examples.
\begin{ex}\label{exemple de X infty}
If $X = \ana^1_k$, let us denote $\eta_{a,r} \in \ana^{1,an}_{k_0}$ the semi-norm $P = \sum a_k(T-a)^k \mapsto \mathrm{max}|a_k|_0 r^k$. As $k_0$ is trivially valued, $\eta_{a,r} \leq 1 \iff \eta_{a,r}(T-a) \leq 1 \iff \eta_{a,r}(T) \leq 1 \iff r\leq 1$. Hence,

\begin{center}
$(\ana^1_k)_\infty = \lbrace x\in \ana^{1,an}_{k_0}, \exists P\in k[T] |P(x)| >1\rbrace = \lbrace \eta_{0,r}, r\in \R_{>1}\rbrace$.
\end{center}

Equally, if $X = \mathbf{G}_{m,k}$ then

\begin{align*}
(\mathbb{G}_{m,k})_\infty &= \lbrace x\in \mathbb{G}^{an}_{m,k_0}, \exists P\in k[T, T^{-1}] |P(x)| >1\rbrace \\
&= \lbrace x\in \mathbb{G}^{an}_{m,k_0}, \mathrm{max} (|T(x)|, |T^{-1}(x)|) >1\rbrace \\
&= \lbrace \eta_{0,r}, r\in \R_{>0}, r\neq 1\rbrace.
\end{align*}
\end{ex}

Now that we have defined $X_\infty$, we may define the $k_{hyb}$-analytic space that will be quotiented to get the compactification.

\begin{defi}
Let $X$ be a $k$ variety, then we define
\begin{center}
$X^+ := X^{hyb} \backslash X^\beth$.
\end{center}
It is an open of $X^{hyb}$ and so it is a $k_{hyb}$-analytic space. One may remark that $X^+_0 = X_\infty$.
\end{defi}

We also have some results about the morphisms.

\begin{prop}
Let $X, Y$ be two $k$-varieties and $f : X \rightarrow Y$ a proper morphism. Then the analytification $f^{hyb} : X^{hyb} \rightarrow Y^{hyb}$ is proper and restricts to a morphism $f^+ : X^+ \rightarrow Y^+$.
\end{prop}

For the proof, we may refer to the proposition 4.2 of Poineau \cite{PoineauCompactificationHybride}.

Finally, to define the hybrid compactification, it remains to define an equivalence relation.

\begin{defi}
Let $X$ be a $k$-variety and let $x,y \in X^{hyb}$. We say that $x,y$ are equivalent through the flow if $T(x) = T(y)$ and we denote it $x \Phi y$.
\end{defi}

\begin{lem}
Let $X$ be a $k$-variety, then $\Phi$ is an equivalence relation.
\end{lem}
For the proof, we may refer to the lemma 4.8 of \cite{PoineauCompactificationHybride}.

We can now define the hybrid compactification :
\begin{defi}\label{immersion de Xan dans compactification}
Let $X$ be a $k$-variety, the set
\begin{center}
$X^\urcorner := X^+/\Phi$
\end{center}
is called the valuative compactification of $X$ or hybrid compactification. We equip this set with the quotient topology, so in particular the image of a subset $V$ of $X^+$ is open if and only if $T(V)$ is open in $X^+$.

Let us denote $q : X^+ \rightarrow X^\urcorner$ the quotient application. 

Let us denote $\delta X := q(X_\infty)$ the boundary of the compactification.

Finally, let us denote $i$ the following immersion:
\begin{center}
$i : (X \otimes_k \hat{k})^{an} = X_1^+ \rightarrow X^+ \rightarrow X^\urcorner$,
\end{center} 
where $\hat{k}$ is the completion of $k$.
\end{defi}

\begin{prop}\label{prop section continue partie archi} Lemma 4.11 and 4.15 of \cite{PoineauCompactificationHybride}.

Let $k$ be a field equipped with a non trivial absolute value. Let $X$ be a $k$ variety, then we may consider its analytification $X^{an}$ with respect to this norm on $k$.

And $i : X^{an} \rightarrow X^\urcorner$ is an homeomorphism into its image.
\end{prop}

\begin{ex}
We saw at the example \ref{exemple de X infty} that $(\ana^1_k)_\infty = \lbrace \eta_{0,r}, r>1\rbrace$, so $\delta (\ana^1_k)$ is a unique point. And $(\mathbb{G}_{m,k})_\infty = \lbrace \eta_{0,r}, r>0, r\neq 1\rbrace$, so $\delta (\mathbb{G}_{m,k})$ is two points. 
\end{ex}

Finally, $X^\urcorner$ has many topological properties.
\begin{prop}
Let $X$ be a $k$-variety. Then, $X^\urcorner$ is Hausdorff and compact. Furthermore, $X^{an}$ is dense in $X^\urcorner$. 

Furthermore, $X^\urcorner$ is locally path-connected and if $X$ is connected, $X^\urcorner$ is path-connected.

If $k$ is countable, then $X^\urcorner$ is metrizable.
\end{prop}
For the proofs, we may refer to the propositions 4.16, 4.22 and 4.23, to the theorem 4.19 and to the lemma 4.20 of \cite{PoineauCompactificationHybride}.

\subsection{Graded reduction and maximum modulus principle}
In this section we use M. Temkin's article \cite{TemkinReductionGradueeII}. We want to use the reduction of affinoid spaces but not necessarily with strictly affinoid spaces. To deal with the reduction in this case, Temkin used graded reductions instead of the classical ones. We recall his definitions here in order to prove a maximum modulus principle over non-archimedean spaces. 

First let us recall the definition of graded rings.
\begin{defi}
Let $G$ be an abelian group, we write it multiplicatively. A $G$ graded ring is a ring $A$ with a unity $1$ 
\begin{center}
$A = \bigoplus_{g\in G} A_g$
\end{center}
such that $A_hA_g \subset A_{gh}$ and $1 \in A_1$. 

A homomorphism of $G$ graded ring is a homomorphism $\phi : A\rightarrow B$ such that $\phi(A_g) \subset B_g, \forall g\in G$.

A graded field is a graded ring with no non-zero homogeneous ideals. It is equivalent to the condition that all non-zero homogeneous element of $A$ are invertible. 

A graded ring $A$ is called local graded ring if it has a unique maximal homogeneous ideal $m$. 

Let $K$ be a graded field. We say that a graded subring $A\subset K$ is a graded valuation ring if it is a local graded ring whose quotient field is $K$ and for any non-zero homogeneous $f\in K$, at least one of the two elements $f, f^{-1}$ belongs to $A$. Any local graded ring $B\subset K$ is dominated by a graded valuation subring of $K$. 

Let $k\subset K$ be two graded fields. We denote by $\p_{K/k} = \p_K$ the set of all graded valuation rings $\mathcal{O}$ such that $k\subset \mathcal{O}$ and the quotient field of $\mathcal{O}$ is $K$. One endows $\p_K$ with the weakest topology with respect to which the set $\lbrace \mathcal{O}\in \p_K, f\in \mathcal{O}\rbrace$ is open for any homogeneous element $f\in K$.

Given $X\subset \p_K$ and $A\subset K$, we define $X\lbrace A\rbrace = \lbrace \mathcal{O}\in \p_K, f\in \mathcal{O}$ for all homogeneous $f\in A\rbrace$.
\end{defi}

Let $k$ be a non-archimedean field, possibly trivially valued.

Let us define some graded rings constructed with Banach $k$-algebras.

\begin{defi}
Let $B$ be a Banach $k$-algebra. Let us denote by $\rho_B$ the spectral semi-norm of $B$. Then, for all positive real number $r$, we set $B_r^\circ = \lbrace x\in B, \rho_B(x)\leq r\rbrace, B_r^{\circ\circ} = \lbrace x\in B, \rho_B(x) < r\rbrace$ and $\tilde{B}_r = B_r^\circ/B_r^{\circ\circ}$.

Let 
\begin{center}
$\tilde{B} = \bigoplus_{r>0} \tilde{B}_r$.
\end{center}
We call it the reduction ring of $B$. 

Any bounded homomorphism $B \rightarrow D$ of Banach $k$-algebras induces a homomorphism $\tilde{B} \rightarrow \tilde{D}$ of $\R_{>0}$ graded $\tilde{k}$-algebras.
\end{defi}

This allows Temkin to define a reduction for the good germs.

\begin{defi}\label{Definition chi_x}
The germ corresponding to a punctural $k$-analytic space $(X,x)$ is denoted $X_x$. A germ is called good if the point $x$ admits an affinoid neighborhood in $X$. 

Let $(Y,x)$ be a punctural $k$-analytic space such that the germ is good. Let $X= M(A)$ be an affinoid neighborhood of $x$, we have a bounded morphism $\chi_x : A \rightarrow \mathcal{H}(x)$, and we set 
\begin{center}
$\widetilde{X_x} = \p_{\widetilde{\mathcal{H}(x)}}(\widetilde{\chi_x}(\tilde{A}))$.
\end{center}
It does not depend on the choice of the neighborhood $X$. 

If $f = (f_1, \cdots, f_n) \in \mathcal{O}_{X,x}$ and $r = (r_1, \cdots, r_n) \in \R^*_+$, then we define $X_x\lbrace r_i^{-1}f_i\rbrace$ as the germ corresponding to the punctured space $(Z,x)$ with $Z = \lbrace y \in X, |f_i(y)| \leq r_i\rbrace$.
\end{defi}

\begin{rem}
Temkin actually defines the reduction of germs for all germs, but we only need the definition in this case.
\end{rem}

We will need the following lemma of Temkin (lemma 4.4 \cite{TemkinReductionGradueeII}).
\begin{lem}\label{lemma Temkin}
Given a good germ $X_x$ and a family $f = (f_1, \cdots, f_n)\subset \mathcal{O}_{X,x}$, one has $\widetilde{X_x\lbrace r_i^{-1}f_i\rbrace} = \widetilde{X_x}\lbrace \tilde{f_i}\rbrace$ where $r_i = |f_i(x)|$ and $\tilde{f_i}$ is the image of $f_i$ in $\widetilde{\mathcal{H}(x)}$.
\end{lem}

In the case where a point of an affinoid space is in the interior, it is easy to compute this reduction and this will allow us to get the maximum modulus principle.
Let us recall the definition of the interior of an affinoid space of Berkovich (\cite{BerkovichLivre}, Definition 2.5.7).

\begin{defi}
Let $k$ be a non-archimedean value field and $X = M(\mathcal{A})$ be a $k$-affinoid space. To every point $x\in  X$, one has an associated bounded $k$-homomorphism $\chi_x : \mathcal{A} \rightarrow \mathcal{H}(x)$ as in definition \ref{Definition chi_x}. It is called inner if there is a admissible epimorphism $\pi : k\lbrace r_1^{-1}T_1, \cdots, r_n^{-1}T_n\rbrace \rightarrow \mathcal{A}$ such that $|\chi_x(\pi(T_i))|_{\mathcal{H}(x)} < r_i, \forall 1\leq i \leq n$. 

The interior of $X$, denoted $\Int(X)$ is the set of points corresponding to inner bounded homomorphisms $\chi_x : \mathcal{A} \rightarrow \mathcal{H}(x)$.
\end{defi}

\begin{prop}\label{reduction is all of p when the point is in the interior}
Let $k$ be a non-archimedean value field and $X$ be a $k$-affinoid space. Let $x \in X$ be a point of the interior $\Int(X)$. Then, one has $\widetilde{X_x} = \p_{\widetilde{\mathcal{H}(x)}}$
\end{prop}
\begin{proof}
One may use Theorem 5.2 (\cite{TemkinReductionGradueeII}) or use a direct argument. Let us denote $X = \mathcal{M}(A), \chi_x : A \rightarrow \mathcal{H}(x)$. By hypothesis, one has a admissible epimorphism $\pi : k\lbrace r_1^{-1}T_1, \cdots, r_n^{-1}T_n\rbrace \rightarrow \mathcal{A}$ such that $|\chi_x(\pi(T_i))|_{\mathcal{H}(x)} < r_i, \forall 1\leq i \leq n$. 

Since $\widetilde{X_x} = \p_{\widetilde{\mathcal{H}(x)}}\lbrace \widetilde{\chi_x}(\widetilde{\pi}(\widetilde{k\lbrace r_1^{-1}T_1, \cdots, r_n^{-1}T_n\rbrace})) \rbrace$, one needs to check that 
\begin{center}
$\widetilde{\chi_x}(\widetilde{\pi}(\widetilde{k\lbrace r_1^{-1}T_1, \cdots, r_n^{-1}T_n\rbrace})) = 0$. 
\end{center} 
For all $i, \widetilde{\chi_x}(\widetilde{\pi})(\widetilde{r_i^{-1}T_i}) \in \widetilde{\mathcal{H}(x)}_{r_i} $ and since $|\chi_x(\pi(T_i))|_{\mathcal{H}(x)} < r_i$, one gets $\widetilde{\chi_x}(\widetilde{\pi})(\widetilde{r_i^{-1}T_i}) = 0$.

Finally, $\widetilde{X_x} = \p_{\widetilde{\mathcal{H}(x)}}$.
\end{proof}

\begin{prop}\label{elements dans tous les anneaux de valuation sont les inversibles}
Let $k$ be a graded field and $K$ a graded field over $k$. Let us recall that $\p_K$ is the set of all graded valuation rings $\mathcal{O}$ with $k \subset \mathcal{O}$. Given a graded valuation ring $\mathcal{O}$, one denotes by $m_\mathcal{O}$ its unique maximal ideal. 

Then, for any non-zero homogeneous element $f\in K$ such that $f \in \bigcap_{\mathcal{O} \in \p_K} \mathcal{O}$, then
\begin{center}
$f \in \bigcap_{\mathcal{O} \in \p_K} \mathcal{O}\backslash m_\mathcal{O}$ and $f^{-1} \in \bigcap_{\mathcal{O} \in \p_K} \mathcal{O}$.
\end{center}
\end{prop}

\begin{proof}

Let $x \in K$ be non-zero homogeneous element such that there exists $\mathcal{O} \in \p_K$ with $x\in m_\mathcal{O}$. Then $x$ is not invertible in $\mathcal{O}$. Hence, $x$ is not invertible in $A$ the subgraded ring of $K$ generated by $x$ and $k$. This implies that $x^{-1}$ is not invertible in $B$ the subgraded ring of $K$ generated by $x^{-1}$ and $k$. Let us denote $\mathfrak{p}$ a prime ideal of $B$ containing $x^{-1}$ and $B_\mathfrak{p}$ the localization with respect to this ideal. This is a local subgraded ring of $K$ with $x^{-1}$ in its maximal ideal. By Zorn lemma, there is a maximal subring of $K$ respecting this condition. By A. Ducros (\cite{DucrosStructureCourbeAnalytique}, 2.2.40.1) this gives a valuation ring $\mathfrak{B}$ containing $k$ and $x^{-1}$ is in its maximal ideal and so, $x \notin \mathfrak{B}$. This concludes the proof. 
\end{proof}

This allows us to get a maximum modulus principle.
\begin{thm}\label{maximum modulus}
Let $k$ be a non-archimedean valued field and $X$ be a connected $k$-affinoid space such that $\Int(X) = X$. Let $f \in \mathcal{O}(X)$ such that there exists $x \in X$ with $|f(x)| = r = \max_{y\in X} |f(y)|$, then for all $y\in X, |f(y)| = r$.
\end{thm}

\begin{proof} 

First, if $r= 0$, then the result is clear. Let us now assume that $r>0$.

It is enough to prove it on a small neighborhood of $x$ since $X$ is connected, and so we can assume that $f$ does not vanish on $X$ and $f^{-1} \in \mathcal{O}(X)$. By proposition \ref{reduction is all of p when the point is in the interior}, one has $\widetilde{X_x} = \p_{\widetilde{\mathcal{H}(x)}}$. By lemma \ref{lemma Temkin} (lemma 4.4 of \cite{TemkinReductionGradueeII}), $\widetilde{X_x\lbrace r^{-1}f\rbrace} = \widetilde{X_x}\lbrace \tilde{f}\rbrace$ but by hypothesis $\widetilde{X_x\lbrace r^{-1}f\rbrace} = \widetilde{X_x}$ so $\tilde{f} \in \widetilde{\mathcal{H}(x)}$ is in all graded subvaluation ring of $\widetilde{\mathcal{H}(x)}$. Hence, by proposition \ref{elements dans tous les anneaux de valuation sont les inversibles}, $\tilde{f}$ is in no maximum ideal of those graded rings which implies that $\tilde{f}^{-1}$ is in all graded valuation rings. Applying lemma \ref{lemma Temkin} (4.4 of \cite{TemkinReductionGradueeII}), one gets
\begin{center}
$\widetilde{X_x\lbrace r f^{-1}\rbrace} = \widetilde{X_x}\lbrace \tilde{f}^{-1} \rbrace = \widetilde{X_x}$.
\end{center}
So, by Theorem 4.5 of \cite{TemkinReductionGradueeII} $X_x\lbrace r f^{-1}\rbrace$ and $X_x$ define the same subdomain of $X_x$. Hence, on a neighborhood $U$ of $x$, for all $y\in U, |f^{-1}(y)|\leq r$ and so $|f(y)|=r$. 
\end{proof}

\begin{rem}
This result also follows from Theorem 1.5 of Gubler-Rabinoff (\cite{GublerRabinoffPSHFunctions}) which states that the local maximum principle stands for all plurisubharmonic functions on an analytic space. Since for any $f\in \mathcal{O}(X), \log |f|$ is a plurisubharmonic function, their theorem is more general result. The two results were proved independently and the two proofs differ in their construction. In their proof, they use the connectedness by curves and prove the result on curves, which by construction comes from graph theory.
\end{rem}

\section{Non-archimedean Julia set}\label{sectionJuliaSet}
In this section, we will recall some results on the Berkovich Julia sets and on certains periodic points. The last part of this section is dedicated to showing how the Julia sets behave with field extensions.

\subsection{First definitions}

Let us first recall some basic definitions about the Julia set and periodic points. For more details, we refer the reader, for example, to the book of J. Milnor \cite{MilnorDynamicsOneComplexVariable} in the complex setting and to the book of R. Benedetto \cite{BenedettoLivre} in the non-archimedean setting.

In complex dynamics, the Julia set is a compact subset of $\p^1(\C)$ which characterises the set where the dynamics is chaotic. Its complement, the Fatou set is where the dynamics behaves nicely. There are a lot of equivalent definitions of these sets. Let us give a few. 

\begin{defi}
Let $f$ be a rational map of degree $d \geq 2$ over $\C$. We say that a point $x\in \p^1(\C)$ is periodic if there exists $n\in \N$, such that $f^{\circ n}(x) =x$. We say that $x$ is a :
\begin{itemize}
\item repelling if $|(f^{\circ n})'(x)|>1$,
\item indifferent if $|(f^{\circ n})'(x)| = 1$ and parabolic if $(f^{\circ n})'(x)$ is a root of unity,
\item attracting if $|(f^{\circ n})'(x)|<1$.
\end{itemize} 

Let us define the Fatou set $\mathcal{F}_f$ as the biggest subset of $\p^1(\C)$ on which the family $(f, f^{\circ 2}, f^{\circ 3}, \cdots)$ is normal. We define the Julia set as $J_f = \p^{1}(\C) \backslash \mathcal{F}_f$.

The Julia set can be caracterized as the following set :
\begin{center}
$\lbrace x\in \p^1(\C), \forall x\in U$ open $\bigcup_{n\in \N} f^{\circ n}(U)$ omits at most 2 points$\rbrace$.
\end{center}

The Julia set is also the support of the equilibrium measure, $\mu_f$, of $f$. It is the only probability measure that verifies the two following conditions :
\begin{itemize}
\item $f^*\mu = d\mu$,
\item the support of $\mu$ is disjoint from the exceptionnal set,
\end{itemize}
where the exceptionnal set is 
\begin{center}
$\mathrm{Exc}(f) := \lbrace a \in \p^1(\C), f^{-1}(a) = \lbrace a \rbrace \rbrace$
\end{center}
and contains at most two points.

This equilibrium measure can be obtained by taking the limit as $n \rightarrow \infty$ of the pull-backs of any Dirac measure, supported on a point not in the exceptionnal set, by $f^{\circ n}$. 
\end{defi}
And we have some links between periodic points and the Julia set. 
\begin{prop} 
Any repelling point is a point in the Julia set, any attracting point is a point in the Fatou set. Any parabolic point is in the Julia set. 

Furthermore, the repelling points are dense in the Julia set. 
\end{prop}

Over any non-archimedean complete valued field, one wishes to extend this notion. One can define the repelling, indifferent and attracting points in the same way. One can also define the classical Fatou set, $\mathcal{F}_{f,classical}$, as  the biggest subset of $\p^1(\C)$ on which the family $(f, f^{\circ 2}, f^{\circ 3}, \cdots)$ is equicontinuous and the classical Julia set, $J_{f, classical}$ as its complement in $\p^1(k)$.

\begin{prop}
Any repelling point is in the classical Julia set and any indifferent or attracting point is in the Fatou set. 
\end{prop}

Unfortunately, in some cases the classical Julia set can be empty.
\begin{ex}
Let $k$ a complete, non-archimedean, valued field and $f(z) = z^d$ where $d\geq 2$. Then, $J_{f, classical} = \emptyset$. A reason is that in this case, all the periodic points are indifferent or attracting if the residual characteristic of $k$ divides $d$.
\end{ex}

An answer to this problem is to work over the Berkovich space $\p^{1,an}_k$ instead of $\p^1(k)$. 
\begin{defi}
Any $\eta \in \ana^{1,an}_k$ extends from $k[T]$ to $k(T)$ by $ \eta (\frac{P}{Q})= \frac{\eta (P)}{\eta (Q)}$ and allowing the value $\infty$ if $\eta (Q) = 0$.

Then, any rational map $f$ over $k$ extends to a map $f : \p^{1,an}_k \rightarrow \p^{1,an}_k$ by 
\begin{center}
$f(\eta) := P \mapsto \eta(P\circ f)$ for all $P\in k[T]$.
\end{center}

Furthermore, for all $a\in k, f(\eta_a) = \eta_{f(a)}$. 

Finally, one can define the Berkovich Julia set as :
\begin{center}
$J_f := \lbrace x\in \p^{1,an}_k, \forall x\in U$ open $\bigcup_{n\in \N} f^{\circ n}(U)$ omits a finite number of points$\rbrace$.
\end{center}
Then, it is a compact subset of $\p^{1,an}_k$ and $J_f \cap\p^1(k) = J_{f,classical}.$

\cite{FavreRiveraLetelierBrolinEquidistribution} defined an equilibrium measure of $f, \mu_f$ on $\p^{1,an}_k$ which is a probability measure supported on the Julia set, invariant by $f$.

Finally, one can consider the periodic points of $f$ in $\p^{1,an}_k$. We say that a point $\eta \in \p^{1,an}_k$ is periodic if $f^n(\eta) = \eta$. If $\eta \notin \p^1_k$, we say that 
\begin{itemize}
\item $\eta$ is repelling if the local degree of $f^n$ at $\eta$ is $\geq 2$,
\item $\eta$ is indifferent if the local degree of $f^n$ at $\eta$ is 1.
\end{itemize}
The repelling points are in the Julia set and the type III and IV indifferent point are in the Fatou set.
\end{defi}
\subsection{Properties of the Berkovich Julia set}

First let us study how the Julia set behaves under field extension, which will allow us to assume that the field we consider are algebraically closed. 

\begin{lem}\label{Taking a field extension does not change the julia set}
Let $f$ be a rational map of degree $d$ defined over a non-archimedean complete valued field $k$. Let $K$ be a field extension of $k$. Let us denote $J_k$ and $J_K$ the Julia set defined over $k$ and $K$, respectively. Let us denote $\pi : \p^{1,an}_K \rightarrow \p^{1,an}_k$ the projection. Then

\begin{center}
$\pi(J_K) = J_k$.
\end{center}
\end{lem}

\begin{proof}
Since $\pi$ is continuous and $J_K$ is compact, so is $\pi(J_K)$.

Firstly, we show that $\pi(J_K) \subset J_k$. Let $x \notin J_k$, then there exists $x \in U$ an open subset of $\p^{1,an}_k$ such that $\bigcup f^n(U)$ omits an infinite number of points of $\p^{1,an}$. Then $\bigcup f^n(\pi^{-1}(U))$ omits an infinite number of points. So, if $y \in \pi^{-1}(x)$, then $y \notin J_K$ and so $x\notin \pi(J_K)$. 

Secondly, let us denote $\mu_{f,k}$ and $\mu_{f,K}$ the equilibrium measures of $f$ seen as a rational map defined over $k$ and $K$, respectively. Then, we have $\pi_* \mu_{f,K} = \mu_{f,k}$. (We check that $f_*(\pi_* \mu_{f,K}) = \pi_* \mu_{f,K}$ and the support of $\pi_* \mu_{f,K}$ is clearly disjoint from the exceptional set which characterizes the equilibrium measure.)

So, we have $\mu_{f,k} (\pi(J_K)) = \mu_{f,K}(J_K) =1$. Since $\pi(J_K)$ is closed, it gives $J_k \subset \pi(J_K)$.
\end{proof}

\begin{rem}
We will later show that if $K$ is the completion of an algebraic extension of $k$, then $\pi^{-1}(J_k) = J_K$ (proposition \ref{Taking an algebraic field extension does not change the julia set}) and if $K$ and $k$ are both algebraically closed, then there exists a continuous morphism : the canonical section, $\sigma_{K/k} : \p^{1,an}_k \rightarrow \p^{1,an}_K$, such that $\sigma_{K/k}(J_k) = J_K$ (proposition \ref{Taking the Julia set over a transcendental extension does not change it}). 
\end{rem}
%\begin{lem}\label{repelling points of type I are dense in the Julia Set}
%Let $f$ be a rational map defined over a non-archimedean complete value field $k$. Let us denote $R = \lbrace x \in \p^1(\overline{k}) | \exists n, f^n(x)$ is a repelling point $\rbrace$. If $R$ is not empty, then $\overline{R} = J_k$, the Julia set of $f$ defined in $\p^{1,an}_k$.
%\end{lem}
%
%\begin{proof}
%By lemma \ref{Taking a field extension does not change the julia set}, one can assume that $k$ is algebraically closed. 
%
%Since $R \subset J_k$, we have $\overline{R} \subset J_k$. By Theorem 8.15 of Benedetto (\cite{BenedettoLivre}), if $f^{-1}(\overline{R}) \subset \overline{R}$, then $J_k \subset R$.
%
%Let $x \in f^{-1}(\overline{R})$ and $x \in U$ open. Since $f$ is quasi-finite, $f : \p^{1,an}_k \rightarrow \p^{1,an}_k$ is open (\textcolor{red}{Mettre ref}) and $f(U)$ is open. Since $f(x) \in f(U) \cap \overline{R}$, there exists $r \in R \cap f(U)$ and so $f^{-1}(r) \in U \cap R$ so $x \in \overline{R}$.
%\end{proof}
%
%It is already known (\textcolor{red}{A enveler})
\begin{lem}\label{classical Julia set is dense}
Let $f$ be a rational map defined over a non-archimedean complete valued field $k$. If the classical Julia set $J_{f,classical}$ is non-empty, then it is dense in the Julia set $J_f$.
\end{lem}
\begin{proof}

By Theorem 8.15 of Benedetto (\cite{BenedettoLivre}), if $f^{-1}(\overline{J_{f,classical}}) \subset \overline{J_{f,classical}}$, then $J_f \subset \overline{J_{f,classical}}$. So it suffices to check this condition.

Let $x \in f^{-1}(\overline{J_{f,classical}}) $ and $x \in U$ open. Since $f$ has finite fibers, $f : \p^{1,an}_k \rightarrow \p^{1,an}_k$ is open (lemma \ref{morphisme fini sur meme dimension est ouvert}) and $f(U)$ is open. Since $f(x) \in f(U) \cap \overline{J_{f,classical}}$, there exists $y\in f(U)$ such that $y \in J_{f,classical}$. Hence, there exists $r \in U$ with $f(r) = y \in J_{f,classical}$, so $r \in J_{f,classical} \cap U$.
\end{proof}

In fact, when the classical Julia set is compact in $\p^1(\overline{k})$, then the classical Julia set is equal to the Julia set. 

We will need the following lemma :
\begin{lem}\label{compact of k are compact in berkovich}
Let $k$ be a complete, algebraically closed valued field. Let $E$ be a subset of $\p^1(k)$. Then the following are equivalent :
\begin{itemize}
\item $E$ is a compact subset of $\p^1(k)$,
\item $\overline{E}^{\p^{1,an}_k} = E$ in $\p^{1,an}_k$.
\end{itemize}
\end{lem}
\begin{proof}
If $k = \C$, then $\p^{1,an}_k = \p^1(k)$ and the result is clear.

Let us now assume that $k$ is non-archimedean, then $\p^1(k)$ is not compact, hence we may assume that $\infty \notin E$.

$\implies$ Let us first show that $E \subset \p^{1,an}_k$ is compact. Let $U_i$ be an open covering of $E$ in $\p^{1,an}_k$. Then $U_i\cap k$ is an open covering of $E$ in $k$. So, me way find a finite subcovering which will be a covering of $E$ in $\p_k^{1,an}$. Hence, $E$ is compact in $\p_k^{1,an}$.

Since the Berkovich spaces over a field are Hausdorff, we get that $E$ is closed in $\p^{1,an}$.

$\Longleftarrow$ Let $U_i$ be an open covering of $E$ in $k$, we may assume that all $U_i$ are open balls in $k$ : $U_i = B(a_i, r_i)$. 

By hypothesis $E$ is closed in $\p^{1,an}_k$ so it is compact in $\p^{1,an}_k$. Since, $B^{an}(a_i, r_i)$ is an open covering of $E$ in $\p^{1,an}_k$ we might find a finite covering $B^{an}(a_j, r_j)$. So, $B(a_j, r_j)$ is a finite subcovering of $E$ in $k$, so $E$ is compact in $\p^1(k)$. 
\end{proof}

\begin{lem}
Let $k$ be a non-archimedean, complete, valued field and $f \in \Rat_d(k)$ of degree $\geq 2$. Then, the following are equivalent :
\begin{itemize}
\item $J_f = J_{f, classical}$
\item $J_{f,classical}$ is a non-empty, compact in $\p^1(\overline{k})$.
\end{itemize}
\end{lem}

We also have some results about the density of repelling points. 

\begin{prop}\label{Repelling points are dense}(Theorem 10.71, \cite{BakerRumelyBook})

Let $K$ be a non-archimedean, algebraically closed, valued field and $f$ be a rational function of degree $d \geq 2$ defined over $K$. Then, the repelling points of $f$ in $\p_K^{1,an}$ are dense in $J_f$.
\end{prop}

And J.-P. Bézivin also proved the following : 

\begin{thm}(Théorème 3, \cite{BezivinPointsRepDenses})

Let $K$ be a non-archimedean, algebraically closed, valued field and $f$ be a rational function of degree $d \geq 2$ defined over $K$. If $f$ has a repelling point defined over $\p^1(K)$, then the repelling points are dense in the classical Julia set. Furthermore, every point of the classical Julia set is an accumulation point of repelling periodic points. 
\end{thm}

Since the classical Julia set is dense in the Berkovich Julia set, one gets the following result.

\begin{prop}\label{repelling points of type I are dense in the Julia Set}
Let $K$ be a non-archimedean, algebraically closed, valued field and $f$ be a rational function of degree $d \geq 2$ defined over $K$. If $f$ has a repelling point defined over $\p^1(K)$, then the type I repelling points are dense in the Berkovich Julia set.
\end{prop}

\subsection{Properties of periodic points}

In this section, we will mostly recall some results of J. Rivera-Letelier (\cite{RiveraLetelierEspaceHyperboliqueEtDynamique}, \cite{RiveraLetelierPeriodicPointsOfRationalMaps}) and R. Benedetto (\cite{BenedettoLivre}). We will sometimes use them in a more general setting.

Firstly, one has a bound on the number of attracting points if the residue characteristic is $0$ or bigger than the degree.

\begin{thm}\label{NumberAttractingPointsBounded}(Theorem 10.17, \cite{BenedettoLivre}) 

Let $k$ be a non-archimedean complete, valued field and denote by $p$ the characteristic of its residue field. Let $f$ be a rational function of degree $d\geq 2$. Suppose $p = 0$ or $p>d$, then there can be at most $2d-2$ attracting periodic cycles. 
\end{thm}

Even if in the archimedean case, the number of indifferent point is bounded, it is not true in the non-archimedean case. In fact, Rivera-Letelier proved the following result : 

\begin{prop}(Proposition 5.8 \cite{RiveraLetelierEspaceHyperboliqueEtDynamique})

Let $\C_p$ be the completion of the algebraic closure of $\Q_p$ and $f \in \Rat_d(\C_p)$ with $d\geq 2$. Then the following are equivalent :
\begin{itemize}
\item $f$ has an indifferent periodic point in $H_{\C_p}$,
\item $f$ has infinitely many indifferent periodic points in $H_{\C_p}$,
\item $f$ has an indifferent periodic point in $\C_p$,
\item $f$ has infinitely many indifferent periodic points in $\C_p$.
\end{itemize}
\end{prop}

R. Benedetto generalized this result for $k$ a non-archimedean valued field whose residue field is algebraic over a finite field.

In the general case, a rational map can have infinitely many non-repelling periodic points. Rivera-Letelier proved the following results in the case $k = \C_p$ but it remains true in all generality. 

\begin{prop}\label{repelling in H_k implies infinitely many}(Corollaire 5.6, \cite{RiveraLetelierEspaceHyperboliqueEtDynamique})

Let $f$ be a rational map of degree $d\geq 2$ over a complete, non-archimedean valued field $k$. Then, if $f$ has a repelling periodic point in $H_k$, then $f$ has infinitely many non-repelling periodic points in $\overline{k}$, the algebraic closure of $k$.
\end{prop}

\begin{proof}
Let us give the idea of the proof. Since $f$ has a repelling fixed point in $H_k$, it is a type II point (see for example Theorem 8.7 of \cite{BenedettoLivre}). Hence, we may assume that it is $\eta_{0,1}$ the Gauss point. The fact that it is a repelling periodic point means that $\tilde{f}$ has degree bigger or equal to 2. Hence, it has infinitely many periodic points in $\p^1(\overline{k})$ and infinitely many of them are not bad directions in the sense of Theorem 7.34 of \cite{BenedettoLivre}. Hence, we have infinitely many balls $D(a,1) = \lbrace x\in k, |x-a|<1\rbrace$ that satisfies $(f^n)(D(a,1)) \subset D(a,1)$ with $|a| = 1$. Since $f^n(T) \in k(T)$, we have $f^n(T) -T = \frac{P_1(T)}{P_2(T)} \in k(T)$ and $\frac{P_1}{P_2}(D(a,1)) \subset D(0,1)$. It remains true with $D^{an}(a,1) = \lbrace x\in \p^{1,an}_k, x(T-a) <1\rbrace$. 

Since the absolute value of a rational map on the projective Berkovich line is constant outside of the convex hull of its zeros and poles, either for all $x \in D^{an}(a,1), |\frac{P_1}{P_2}(x)| = constant < 1$ or there is an $x \in D^{an}(a,1)$ such that $|\frac{P_1}{P_2}(x)|=0$ and it will be a type I point of $\p^{1,an}_k$. Since $|\frac{P_1}{P_2}(\eta_{0,1})|=1$, the first case is impossible. 

So, we get a periodic point, since $f^n(B^{an}(z,1)) \subset B^{an}(z,1)$ it is non-repelling.   
\end{proof}

\begin{prop}\label{2 attracting points implies infinitely many}(Lemme 6.1 + proof of the Théorème A', \cite{RiveraLetelierEspaceHyperboliqueEtDynamique})

Let $f$ be a rational map of degree $d\geq 2$ over a complete, non-archimedean valued field $k$. If $f$ has an attracting periodic point and at least an other non-repelling periodic point, then it has a repelling periodic point in $H_k$. Hence, it has infinitely many non-repelling periodic points. 
\end{prop} 

In the case where $k = \C_p$ Rivera-Letellier gives a description of the rational map with a finite number of repelling points.

\begin{thm}(Théorème 3, \cite{RiveraLetelierPeriodicPointsOfRationalMaps})

Let $f$ be a rational map of degree $d\geq 2$ over $\C_p$ then the following are equivalent :
\begin{itemize}
\item $f$ does not have any periodic point in $H_p$,
\item $f$ has finitely many non-repelling periodic points in $\p^1(\C_p)$,
\item $f$ has exactly one non-repelling periodic point in $\p^1(\C_p)$.

In this case, $f$ has one attracting periodic point and all the others are repelling.
\end{itemize}
\end{thm}

When the field $k$ is not necessarily $\C_p$ but is the completion of an algebraic closure of a discretely valued field, we have a similar result. 

\begin{prop}\label{if classical julia set compact, then one attracting only}

Let $f$ be a rational map over a complete, algebraically closed, non-archimedean, valued field $\C_v$. Let us assume that there is subfield $K$ which is discretely valued and such that $\overline{K}$ is dense in $\C_v$. Furthermore, assume that $f\in K(z)$. Then, the following are equivalent :
\begin{itemize}
\item The Julia set of $f$ contains only type I and type IV points,
\item $f$ has no periodic point in $H_{\C_v}$.
\end{itemize}

Futhermore in this case, $f$ has one attracting periodic point and all the others are repelling. 

In particular, the following are also equivalent : 
\begin{itemize}
\item The classical Julia set of $f$ is non-empty and compact in $\p^1(\C_v)$,
\item $f$ has no periodic point in $H_{\C_v}$ and no type IV points in its Julia set.
\end{itemize}
\end{prop}
\begin{proof}
For the first equivalence of the proposition, let us assume that the Julia set contains only type I and type IV points.  

If there is a periodic point in $H_{\C_v}$, then it cannot be in the Julia set, because it would be a type II periodic point in the Julia set (see for example Theorem 8.7 of \cite{BenedettoLivre}). So it is in a Fatou connected component $U$. But its boundary is a finite number of type II points that are periodic (see for example Theorem 9.7 of \cite{BenedettoLivre}) and in the Julia set, which is impossible.

Let us now assume that there is no periodic point in $H_{\C_v}$. Let $U$ be a connected component of the Fatou set. Then, by Theorem 11.23 of \cite{BenedettoLivre}, it cannot be a wandering component. Otherwise, it would lie in an attarcting bassin of a Type II Julia periodic point. It cannot be an indifferent component either, because the boundary is composed of type II periodic point. So it is an attracting component but it cannot be a disk, because the boundary of such an attracting component is a type II repelling periodic point so it is a domain of Cantor type. 

There can only be one, otherwise there would be at least two attracting points and so there would be a repelling periodic point in $H_{\C_v}$ by proposition \ref{2 attracting points implies infinitely many}.

So, the Julia set is the boundary of this connected component. Let us assume that there is a type II or type III point in this boundary, then we could write it $\eta_{a,r}, a\in \C_v$. We can assume that the attracting point is $\infty$ and $D^{an}(a,r)$ is an open set disjoint from the attracting bassin of $\infty$. Since the Julia set is of empty interior, $D^{an}(a,r)$ cannot be contained in the Julia set and so there is at least another connected Fatou component, which is impossible. Hence, there is only type I and type IV points in the Julia set.  

The second part of the proof shows that $f$ has only one attracting points and all the others are repelling.

For the second equivalence,

Let us assume that the classical Julia set of $f$ is compact and non-empty. Then, $J_{f,classical} \subset \p^{1,an}_{\C_v}$ is compact. By lemma \ref{classical Julia set is dense}, $J_f$ is the closure of the classical Julia set, hence $J_f = J_{f,classical}$. Hence, the Julia set contains only type I points, so by the first equivalence, there is no periodic points in $H_{\C_v}$ and no type IV points in the Julia sets.

Let us assume that $f$ has no periodic points in $H_{\C_v}$ and no type IV in its Julia set. Then, there is only type I points in the Julia set. So the Julia set is contained in $\p^1(\C_v)$, so it equal to the classical set. Hence, this set is compact. 
\end{proof}

In some cases, there is a more general result :

E. Trucco proved the following when $f$ is tame and a polynomial.
\begin{prop}(Corollary 7.15, \cite{TruccoWanderingFatouComp})

Let $f \in K[z]$ a tame polynomial of degree $d \geq 2$, then the following are equivalent :
\begin{itemize}
\item $J_f = J_{f, classical}$,
\item All the periodic points in $K$ are repelling,
\item There is no critical periodic element in $J_p$.
\end{itemize}  
\end{prop}

\subsection{The Julia set is independant of the field}
In this section we will prove the inverse of the lemma \ref{Taking a field extension does not change the julia set} and see that if we have a field extension $K$ of $k$ and $f$ a rational map defined over $k$, then we can obtain the Julia set $J_K \subset \p^{1,an}_K$ from the Julia set $J_k \subset \p^{1,an}_k$.

\begin{prop}\label{Taking an algebraic field extension does not change the julia set}
Let $f$ be a rational map of degree $d$ defined over a non-archimedean complete valued field $k$. Let $K$ be the completion of an algebraic field extension of $k$. Let us denote $J_k$ and $J_K$ the Julia set defined over $k$ and $K$, respectively. Let us denote $\pi : \p^{1,an}_K \rightarrow \p^{1,an}_k$ the projection. Then

\begin{center}
$\pi^{-1}(J_k) = J_K$.
\end{center}
\end{prop}

To show this we will need the following result. 

\begin{prop}\label{measure on the union of field extension has same support}
Let $(k_i)_{i\in J}$ be a family of complete valued fields and $J$ be an ordered set such that if $i\leq j$ then $k_j$ is a field extension of $k_i$. Let us denote $\pi^{j,i} : \p^{1,an}_{k_j} \rightarrow \p^{1,an}_{k_i}$ the projection for $i \leq j$. 

Let us assume that we have a probability measure $\mu_j$ on $\p^{1,an}_{k_j}$ for all $j$ such that $\pi^{j,i}_*\mu_j = \mu_i$ for all $i\leq j$. Let us denote $S_j$ the support of $\mu_j$ and assume that $(\pi^{j,i})^{-1}(S_i) = S_j$.   

Let us denote $K = \widehat{\mathrm{colim}_{j\in J}k_j}$ and $\pi^j : \p^{1,an}_K \rightarrow \p^{1,an}_{k_j}$ the projection. Then one can define a probability measure $\mu$ on $\p^{1,an}_K$ such that for all $j\in J, \pi^j_*\mu = \mu_j$ and the support $S$ of $\mu$ satisfies $(\pi^j)^{-1}(S_j) = S$ for all $j$. 
\end{prop}
\begin{proof}
Let us denote $A := \lbrace F \in \mathcal{C}^0(\p^{1,an}_K, \R), \exists j\in J, f\in \mathcal{C}^0(\p^{1,an}_{k_j}, \R), F = f\circ \pi^j\rbrace.$ Then, by the theorem of Stone-Weierstrass (see for example the theorem 7.29 in Chapter II of \cite{HewittStrombergRealAndAbstractAnalysis}) $A$ is dense in $\mathcal{C}^0(\p^{1,an}_K, \R)$. One needs to check that if $x\in \p^{1,an}_K$, there exists $F\in A$ such that $F(x) \neq 0$ and if $x\neq y \in \p^{1,an}_K$, there exists $G\in A$ such that $G(x) \neq G(y)$. 

For the first case, since for all $x\in \p^{1,an}_K$ and for all $j \in J$, there exists $f_j \in \mathcal{C}^0(\p^{1,an}_{k_j}, \R)$ with $f_j(\pi^j(x)) \neq 0$ one can take any $f_j$ and consider $F = f_j \circ \pi^j$.

If $x\neq y$ there exists $P \in K[T]$ such that $|P(x)| \neq |P(y)|$. Let $\epsilon < \max(|P(x)|, |P(y)|)$, then there exists $Q \in \mathrm{colim}_{j\in J}k_j[T], ||Q-P|| <\epsilon$ and so $|Q(x)| \neq |Q(y)|$. Furthermore, $Q \in k_j[T]$ for some $j\in J$. Then $\pi^j(x) \neq \pi^j(y)$ and there exists $f_j \in \mathcal{C}^0(\p^{1,an}_{k_j}, \R)$ such that $f_j(\pi^j(x)) \neq f_j(\pi^j(y))$ and $F = f_j \circ \pi^j \in A$ is such that $F(x) \neq F(y)$. 

First we define $\mu$ for $F = f_j  \circ \pi^j \in A$ by 
\begin{center}
$\int_{\p^{1,an}_K} F d\mu = \int_{\p^{1,an}_{k_j}} f_j d\mu_j$.
\end{center}
One needs to check that if $f_j \circ \pi^j = g_i \circ \pi^i$ then $\int_{\p^{1,an}_{k_j}} f_j d\mu_j = \int_{\p^{1,an}_{k_i}} g_i d\mu_i$. We may assume $i \leq j$ and so $f_j \circ \pi^{j,i} = g_i$ and since $\pi^{j,i}_*\mu_j = \mu_i$ the equality stands. 

Now to define it on all $C^0(\p^{1,an}_K, \R)$, one needs to check that if $F_n, G_n \in A^\N$ converge to the same limit, then $\int F_n d\mu = \int G_n d\mu$. But since, $\p^{1,an}_K$ is compact, if $F_n \rightarrow f =  \lim G_n$, then $||F_n -G_n||_\infty \rightarrow 0$ and with $\int_{\p^{1,an}_K} 1 d\mu = 1$, we get $\int_{\p^{1,an}_K} (F_n - G_n) d\mu \rightarrow 0$.

Finally, we only need to show that $S = (\pi^j)^{-1}(S_j)$ for all $j$. 

First let $j \in J$. Then, $\pi^j_*\mu = \mu_j$ and so $\pi^j(S) \subset S_j$ like in the proof on lemmma \ref{Taking a field extension does not change the julia set} and $S\subset (\pi^j)^{-1}(S_j)$.

Then, since $(\pi^{j,i})^{-1}(S_i) = S_j$ for all $i\leq j$, then $(\pi^j)^{-1}(S_j) = (\pi^i)^{-1}(S_i)$ for all $i,j$.

Let $x\in (\pi^j)^{-1}(S_j)$ for all $j$ and let $U$ be a open set of $\p^{1,an}_K$ containing $x$. One needs to show that there exists $f\in \mathcal{C}^0(U, \R)$ with $\int_U fd\mu \neq 0$. It is enough to show it for a basis of neighborhood of $x$ and so we may assume that $U = \lbrace y\in \p^{1,an}_K, r_1 < |P_1(y)| < s_1, \cdots, r_k < |P_k(y)| < s_k \rbrace$ with $P_l \in K[T]$. As previously, in the proof we may assume that $P_l \in \bigcup_{j\in J} k_j$ and even $P_l \in k_j$. Hence, $\pi^j(U) = V$ is open in $\p^{1,an}_{k_j}$. Since $x\in (\pi^j)^{-1}(S_j)$ there exists $f_j \in \mathcal{C}^0(V, \R)$ with $\int_V f_j d\mu_j \neq 0$. 

Hence, with $f = f_j \circ f_j, \int_U fd\mu \neq 0$. 
\end{proof}

We can now prove the proposition \ref{Taking an algebraic field extension does not change the julia set}.
\begin{proof}
First let us assume that the morphism $\pi : \p^{1,an}_K \rightarrow \p^{1,an}_k$ is finite, so $K$ is a finite algebraic extension of $k$.

%Let us denote $\mu_{f,k}$ and $\mu_{f,K}$ the equilibrium measures of $f$ seen as a rational map over $k$ or $K$. Since $\pi$ is finite, one can take the pull-back of $\mu_{f,k}$ by $\pi$ and if $e = \mathrm{deg} \pi$, then $\mu_{f,K} = \frac{1}{e} \pi^*\mu_{f,k}$. We check that $f_*(\frac{1}{e} \pi^*\mu_{f,k}) = \frac{1}{e} \pi^*\mu_{f,k}$ and the support of $\frac{1}{e} \pi^*\mu_{f,k}$ is disjoint from the exceptional set.  

Since $\pi(J_K) = J_k$, one has $J_K \subset \pi^{-1}(J_k)$.
Let $x\in \pi^{-1}(J_k)$ and $U$ an open subset containing $x$. Since $\pi$ is finite, $\pi(U)$ is also open and there exists $g\in \mathcal{C}^0(\pi(U), \R)$ such that $\int_{\pi(U)} g d\mu_{f,k} \neq 0$ and so $\int_U g\circ \pi~ d\mu_{f,K} = \int_{\pi(U)} g d\mu_{f,k} \neq 0$. Hence, $x \in J_K$.

Now, let $K$ be the completion of an algebraic extension of $k$. Let $\alpha$ be an ordinal such that all the algebraic elements of $K$ over $k$ are $x_\beta$ indexed by all $\beta \leq \alpha$.

Then, we will define a tower of field extensions in the following way : 
\begin{itemize}
\item $k_0 = k$,
\item If $\beta \leq \alpha$ is the successor of an ordinal $\gamma$, then $k_\beta = k_\gamma(x_\gamma, x_\beta)$,
\item If $\beta \leq \alpha$ is a limit ordinal, then $k_\beta = \widehat{\mathrm{colim}_{\gamma <\beta}k_\gamma}$.
\end{itemize}
We will show by transfinite induction that $\forall \beta \leq \alpha, \forall \gamma \leq \beta : (\pi^{\beta, \gamma})^{-1}(J_\gamma) = J_\beta$.

For $\beta = 0$, the result is clear. 

Let $\beta$ be an ordinal such that for all $\delta<\beta$ the result is true. 

If $\beta$ is a successor ordinal $\beta = s(\delta)$, then by the beginning of the proof, $J_\beta = (\pi^{\beta, \delta})^{-1}(J_\delta)$ and the result follows for all $\gamma \leq \delta$.

If $\beta$ is a limit ordinal, then $k_\beta = \widehat{\mathrm{colim}_{\gamma <\beta}k_\gamma}$ and we are in the case of the proposition \ref{measure on the union of field extension has same support}. The only thing that we need to check is that the measure constructed in the proof is indeed the equilibrium measure of $f$ seen as a rational map over $k_\beta$. To check that $f_*\mu = \mu$, it is enough to show that for all $g\in \mathcal{C}^0(\p^{1,an}_{k_\beta}, \R)$ with $g = g_\gamma \circ \pi^{\beta, \gamma}$ we have $\int g\circ f d\mu = \int g d\mu$. 

Which is true since 
\begin{align*}
\int g d\mu &= \int g_j \circ \pi^{\beta, \gamma} d\mu_\gamma\\
&= \int g_j \circ \pi^{\beta, \gamma} \circ f d\mu_\gamma\\
&= \int g_j \circ f \circ \pi^{\beta, \gamma} d\mu_\gamma\\
&= \int g\circ f d\mu.
\end{align*}
The equality in the second line comes from the fact that $f_*\mu_\gamma = \mu_\gamma$ and the equality in the line 3 and 4 come from the fact that $f$ commutes with the projection. 

Furthermore since the support is the inverse image of the support of $\mu_\gamma$ for all $\gamma$ it is disjoint from the exceptional set and $\mu = \mu_\beta$. 
\end{proof}

Let us now focus on the case where both $K$ and $k$ are algebraically closed. In this case it will no longer be true that $J_K = \pi^{-1}(J_k)$ as we will see it in the following example.

\begin{ex}
Let $k = \C$ trivially valued and $K = \overline{\C(T)}$ where $\C(T)$ is trivially valued. Then, any $f \in \Rat_d(k)$ has good reduction and $J_k = \eta_{0,1}$. The map $f$ has also good reduction when we see it in $K$ and $J_K = \eta_{0,1}$ but $\pi(\eta_T) = \eta_{0,1}$ so $J_K \neq \pi^{-1}(J_k)$.
\end{ex}
To avoid this, we will recall the notion of universal point as defined by J. Poineau (\cite{PoineauAngelique}, §3) based on the definition of peaked points of V. Berkovich (\cite{BerkovichLivre}, §5.2). It will allow us to have a continuous section $\sigma_{K/k} : \p^{1,an}_k \rightarrow \p^{1,an}_K$.

\begin{defi}\label{definition universal point and continuous section}
Let $k$ be a complete valued field and $X$ be a $k$-analytic space. A point $x\in X$ is said to be universal if for all complete valued extension $L$ of $k$, the norm on the algebra $\mathcal{H}(x) \hat{\otimes} L$ is multiplicative. 

For $L$ a complete valued field extension of $k$, let us denote $X_L := X \hat{\otimes} L$. Then if $x \in X$ is universal it defines a canonical point $\sigma_{L/k}(x)$ which is the point corresponding to the multiplicative norm on $\mathcal{H}(x) \hat{\otimes} L$. 
\end{defi}

\begin{prop}(Corollaire 3.14 \cite{PoineauAngelique})

If $k$ is algebraically closed, $X$ a $k$-analytic space then all points of $X$ are universal.
\end{prop}

\begin{ex}
If $k$ is an algebraically closed, valued field, $L$ is any valued field extension, $X= \ana^{1,an}_k$ and $x\in k$. Then, $\sigma_{L/k}(x)$ is the point in $\ana^{1,an}_L$ corresponding to the point $x$ seen as an element of $L$.

And for all $r>0$, $\sigma_{L/k}(\eta_{x,r})$ is the point $\eta_{x,r} \in \ana^{1,an}_L$.
\end{ex}

\begin{prop}(Corollaire 3.7 \cite{PoineauAngelique})

Let $k$ be a complete valued field and $K$ a field extension of $k$. Let us denote $X_u$ the set of all universal points then, the map $\sigma_{K/k} : X_u \rightarrow X\hat{\otimes} K$ is continuous.

In particular, if $k$ is algebraically closed, then $\sigma_{K/k} : X \rightarrow X\hat{\otimes} K$ is continuous.
\end{prop}

We can now state the result about the Julia set when both $K$ and $k$ are algebraically closed. 

\begin{prop}\label{Taking the Julia set over a transcendental extension does not change it}
Let $k \subset K$ with $k, K$ two complete, algebraically closed, valued fields and $f\in \Rat_d(k)$. Let $\sigma_{K/k} : \p^{1,an}_k \rightarrow \p^{1,an}_K$ the continuous section of the definition \ref{definition universal point and continuous section}.

Then, 
\begin{center}
$\sigma_{K/k}(J_k) = J_K$.
\end{center}
Hence, $J_K$ and $J_k$ are homeomorphic and $\sigma_{K/k} \circ \pi_{K/k} = \mathrm{Id}_{J_K}$.
\end{prop}

\begin{proof}
Let us denote $R_k := \lbrace x\in \p^{1,an}_k, x$ is a repelling point of type I or II$\rbrace$ and $R_K := \lbrace x\in \p^{1,an}_K, x$ is a repelling point of type I or II$\rbrace$, then it suffices to show that $\sigma_{K/k}(R_k) \subset J_K$ and $R_K \subset \sigma_{K/k}(J_k)$ since $R_k$ and $R_K$ are dense in $J_k$ and $J_K$ by proposition \ref{Repelling points are dense}.

Let $x\in R_k$. If $x$ is a type I point, then $x\in k$ and $f^n(x) = x$ and $|(f^n)'(x)|>1$. Hence if we see $x$ as a point in $K$, by the correspondance $x = \sigma_{K/k}(x), x\in R_K$. So, $\sigma_{K/k}(x) \in R_K\subset J_K$.
If $x$ is a type II point, $x = \eta_{a,r}$ and there exists $M \in \GL_2(k)$ such that $M\cdot x = \eta_{0,1}$ and the reduction of $f^M$ is of degree at least 2. Then, all of this remains true, if we see $\eta_{a,r} \in \p^{1,an}_K,$ by the correspondance $\eta_{a,r} = \sigma(\eta_{a,r}), M\in \GL_2(K)$ and $\sigma_{K/k}(x) \in R_K \subset J_K$.

Let $x\in R_K$. If $x$ is a type I point, then $x\in K$ and $f^n(x) = x$ and $|(f^n)'(x)|>1$. Since $k$ is algebraically closed, $x\in k$ and $x\in \sigma_{K/k}(R_k) \subset \sigma_{K/k}(J_k)$. If $x$ is a type II point, $x = \eta_{a,r}$ then there exists $M \in \GL_2(K)$ such that $M\cdot x = \eta_{0,1}$ and the reduction of $f^M$ is of degree at least 2. Furthermore, if $y\in K, |y|= r$ then we may take $M$ to be 
$\begin{pmatrix}
\frac{1}{y}&\frac{-a}{y}\\
0&1
\end{pmatrix}.$

Since $f^M$ has a reduction of degree $\geq 2$, then by the proof of the proposition \ref{repelling in H_k implies infinitely many}, there exists $z\in K$ such that $(f^M)^n(z) = z$ and $|z|\leq 1$. 

Then, $M^{-1}z$ is a periodic point of $f$ and so is a point of $k$. Since $M^{-1}z = yz+a$, one has $|M^{-1}z -a| = |yz| \leq r$. Hence, $\eta_{a,r} = \eta_{M^{-1}z,r}$ and $\eta_{a,r} \in \sigma_{K/k}(\p^{1,an}_k)$. It remains to show that $\eta_{M^{-1}z,r} \in J_k$. Since $\eta_{M^{-1}z,r} = \pi(\eta_{a,r}) \in \pi(J_K)$, we have that $\eta_{M^{-1}z,r} \in J_k$ by lemma \ref{Taking a field extension does not change the julia set}.
\end{proof}
Let us now explain how one can obtain the Julia set over a field extension in all cases.

\begin{cor}\label{Taking the Julia set over any extension does not change it}
Let $k$ be a complete valued field, $K$ be a complete valued field extension. Then, $\widehat{\overline{K}}$ the completion of an algebraic closure of $K$ is a valued field extension of $\widehat{\overline{k}}$. Let us denote $\pi_{K, k}, \pi_{\widehat{\overline{k}}, k}, \pi_{\widehat{\overline{K}}, K}$ the projections and $\sigma_{\widehat{\overline{K}}, \widehat{\overline{k}}}$ the continuous section. 

Then,
\begin{center}
$\pi_{K,k}(J_K) = J_k$ and $\pi_{\widehat{\overline{K}}, K}\circ \sigma_{\widehat{\overline{K}}, \widehat{\overline{k}}}~ (\pi_{\widehat{\overline{k}}, k}^{-1}(J_k)) = J_K$.
\end{center}
\end{cor}

\section{Continuity of the Julia set in the non-archimedean setting}\label{Section continuity in the non archimedean settings}
Let us now study the conditions to have a continuous motion of the non-archimedean Julia set. Our goal is to prove a non-archimedean version of the theorem of \cite{MSS} on the holomorphic motion of the Julia set.
\subsection{A result on the continuity of the Julia set}
\begin{defi}
Let $k$ be a non-archimedean complete value field and $X$ a $k$-analytic space. An analytic family of rational maps of degree $d\geq 2$, parametrized by $X$, is an analytic map $f : X \times_k \p^{1,an}_k \rightarrow X\times_k \p^{1,an}_k$ such that $\pr_X = \pr_X \circ f$ where $\pr_X : X\times_k \p^{1,an}_k$ is the first projection and for all $x \in X, f_x := f(x, \cdot) : \p^{1,an}_{\mathcal{H}(x)} \rightarrow \p^{1,an}_{\mathcal{H}(x)}$ is a rational map of degree $d\geq 2$.

A pair $(x,z) \in X\times_k \p^{1,an}_k$ can be seen as an element $x\in X$ and $z\in \p^{1,an}_{\mathcal{H}(x)}$ and this will be the point of view used in the rest of this article.
\end{defi}

T. Silverman studied the analytical motion of the Julia set and he got the following theorem in \cite{TSilvermanNonArchiLambdaLemma}.
\begin{thm}
Let $K$ be a complete algebraically closed non-archimedean field, let $U \subset \ana^{1,an}_K$ be a connected set, and let $(f_\lambda)_{\lambda \in U}$ be an analytic family of rational maps. Assume that the type I repelling points are dense in $J_{x_0}$ for some $x_0 \in U$. Then, the following are equivalent :
\begin{enumerate}
\item The Julia sets move analytically on $U$,
\item For all $x\in U, f_x$ has no unstably indifferent periodic point and no type I repelling periodic points with multiplicity greater than 1.
\end{enumerate}
\end{thm}

In this result, he focuses on a notion of an analytical motion in a strong sense which gives some restriction on the result : $U$ is a subset of $\ana^{1,an}_K$ and there is an extra condition on the repelling points $x$ of $U$ : they have to be of multiplicity equal to one. It means that if $x\in U$, then all the repelling points have to be in $\mathcal{H}(x)$ and not in its algebraic closure but in general, this field is not algebraically closed. 

For instance, one may look at a family of rational maps on $D^{an}_\C(0,1) = \lbrace x\in \ana^{1,an}_\C, |T(x)| <1\rbrace$ where $\C$ is trivially valued and $\ana^1_\C = \mathrm{Spec} ~\C[T]$. To any, $x\in D^{an}_\C(0,1), f_x : z\mapsto T(x)z^3 +z^2$ where $T(x)$ is the image of $T \in \mathcal{H}(x)$. And as soon as $x \in D^{an}_\C(0,1)$ is not a type I point, the field $\mathcal{H}(f_x) \cong \C((t))$ will not be algebraically closed.

To avoid this problem, we only look at some continuous motion. This will allow us to delete this condition and to work over any $K$-analytic space. We will define a different notion of ,,unstably indifferen'' point that we will call ,,not persistently indifferent''. We will show that the two notions coincide.

\begin{defi}
Let $f$ be an analytic map and $x\in X$. We say that the Julia moves continuously on a neighborhood of $x$ if there exists a neighborhood $U$ of $x$ such that for all $y\in U$ and all family $(y_i)_{i\in I} \in U^I$ with $y$ in the closure of the $(y_i)$, we have :
\begin{center}
$\overline{\lbrace (f, z) \in X\times_k \p^{1,an}_k, \exists i\in I, y_i = f~and~z \in J_f\rbrace} \cap \pr_X^{-1}(y) = \lbrace (f, z) \in X\times_k \p^{1,an}_k, f=y$ and $z \in J_y\rbrace$
\end{center}
where $\overline{E}$ denotes the closure of the set $E$.

This means that one requires
\begin{center}
$\overline{\bigcup J_{y_i}} \cap \pr^{-1}_X(y) = J_y$
\end{center}
where we see the Julia set $J_z$ embedded in $X \times_k \p^{1,an}_k$ in the fiber $\pr^{-1}_X(z)$.
\end{defi}

\begin{rem}
In the case, where the field $k$ is countable $X\times_k \p^{1,an}_k$ is metrizable and this definition is equivalent to the continuity of the family of compact sets $(\lbrace (x, y), y\in J_x\rbrace)_{x \in X}$ with respect to the Haussdorf distance.
\end{rem}

\begin{defi}\label{Definition of the periodic point}
Let $k$ be a non-archimedean complete valued field and $X$ a $k$-analytic space. Let $f$ be an analytic family of rational maps of degree $d \geq 2$ and $n\in \N$. Let $U$ be a small enough open of $X$ such that for all $x\in X, f_x =\frac{a_d(x)z^d+ \cdots + a_0(x)}{b_d(x)z^d+ \cdots +b_0(x)}$ and all the $a_i$ and $b_i$ are in $\mathcal{O}(U)$. Let $V$ be an open set of $X \times_k \p^{1,an}$ such that $\pr_X(V) \subset U$, then $f_x^n(z)-z = \frac{P_1(z)}{P_2(z)} \in \mathcal{O}(V)$. Hence, one can define $C^n_V := \lbrace (x,z) \in V, |f_x^n(z)-z|=0 \rbrace$ the closed analytic space of the fixed points of $f^n$. For each $(x,z)\in C^n_V, z$ can be seen as type I points in $\p^{1,an}_{\mathcal{H}(x)}$.
\end{defi}

\begin{prop}
Let $k$ be a non-archimedean complete valued field and $X$ a $k$-analytic space. Let $f$ be an analytic family of rational maps of degree $d \geq 2$ and $n\in \N$. Let $C^n_V$ be as in definition \ref{Definition of the periodic point}, then 
\begin{center}
If $x\in \Int (X)$, then for all $(x,t) \in \pr_X^{-1}(x), (x,t) \in \Int (C^n_V)$
\end{center}
where $\pr_X$ is the projection from $C^n_V$ to $X$ via the first coordinate.
\end{prop}

\begin{proof}
By the third point of the proposition 2.5.8 of \cite{BerkovichLivre}, one only needs to check that $(x,t) \in \Int (C^n_V/X)$.

Since this map has finite fibers, it is locally finite (lemme \ref{morphisme fini sur meme dimension est ouvert}). So, by the converse implication of the Corollary 2.5.13 of \cite{BerkovichLivre}, $(x,t) \in \Int (C^n_V/X)$. 
\end{proof}

Hence, we will be able to apply the maximum modulus principle (theorem \ref{maximum modulus}) to any point of $C^n_V$ if we suppose that $\Int (X) = X$.
\begin{defi}
A $k$-analytic space will be said without boundary if $\Int (X) = X$.
\end{defi}
It will be the setting in the remaining of this section. Let us now define the notion of persistently indifferent periodic point which will be the main condition to have a continuous motion.

\begin{defi}\label{Definition of persistently indifferent}
Let $k$ be a non-archimedean complete valued field and $X$ a $k$-analytic space. Let $f$ be an analytic family of rational maps and $x \in X$. A periodic point $z \in \p^{1,an}_{\mathcal{H}(x)}$ of $f_x$ of period $n$ of type I is said to be persistently indifferent if :
%there is a neighborhood $U$ of $x$ and a holomorphic map $w: U \rightarrow U\times_k \p^{1,an}_k$ such that $id_U = pr_U\circ w, w(x) = (x,z)$ and $f^n_\lambda(w(\lambda)) = w(\lambda), | (f^n_\lambda)'(w(\lambda))|=1, \forall \lambda \in U$.
\begin{center}
    There is a neighborhood $U$ of $(x,z)$ such that $\forall (\lambda, y) \in U$ if $y$ is a type I periodic point of $f_\lambda$ of period $n$, then $|(f_\lambda^n)'(y)| =1$. 
\end{center}
\end{defi}

\begin{rem}
There is a difference with the archimedean case here. Since the indifferent point can be in the algebraic closure of the residual field of $f_\lambda$, they can have a multiplier strictly greater than 1.
So, we cannot necessarily have an analytic motion of the indifferent point, hence we have a result on a neighborhood of the point instead of following holomorphically the point on a neighborhood.  
\end{rem}

\begin{ex}
Let us assume that $k$ is algebraically closed and $f\in \Rat_d(k)$. Then all indifferent periodic point of $f$ are persistently indifferent.

Let $(f,z) \in (\Rat_d \times \p^1)^{an}(k)$ be an indifferent point. Then, $(f^n)'(z)  \in k$, hence $U = \lbrace (y,t) \in (\Rat_d\times \p^1)^{an}, |(f^n_y)'(t) - (f^n)'(z)|<1\rbrace$ is a neighborhood of $(f,z)$. This implies that if $(y,t)$ is a periodic point of $f_y$ of period $n$, then it is indifferent. So $z$ is persistently indifferent.
\end{ex}

\begin{defi}
Let $K$ be algebraically closed, complete, valued field and $X \subset \ana^{1,an}_K$.
Let us recall the definition of unstably indifferent in the sens of Silverman \cite{TSilvermanNonArchiLambdaLemma}. Let $f : X \times_K \p^{1,an}_K\rightarrow X\times_K\p^{1,an}_K$ be an analytical family of map of degree $d$. Then, for all $x\in X, f_x : \p^{1,an}_{\mathcal{H}(x)} \rightarrow \p^{1,an}_{\mathcal{H}(x)}$ is a rational map of degree $d$. So, one can see $f'_\lambda : \p^{1,an}_{\mathcal{H}(x)} \rightarrow \p^{1,an}_{\mathcal{H}(x)}$ as a rational map of degree $e$. And so one get $f' : X\times_K \p^{1,an}_K \rightarrow X\times_K \p^{1,an}_K$ an analytical family of rational map. Let $(x,t)$ be an indifferent periodic point. By replacing $f$ by some $f^{\circ n}$, we may suppose that it is of period 1. 

Then it is unstably indifferent if $\pr_{\p^{1,an}_K}(f'(x,t)) = \eta_{0,1}$.
\end{defi}

\begin{lem}\label{Unstably and not persistently indifferent is the same in dimension one}
Let $K$ be algebraically closed, complete, valued field and $X \subset \ana^{1,an}_K$. Then a point $(x,t)$ is not persistently indifferent if and only if $(x,t)$ is unstably indifferent in the sense of Silverman.
\end{lem}
\begin{proof}

Let $(x,t)$ be a periodic point, then we may find an open neighborhood such that $C_V$ is well defined (see definition \ref{Definition of the periodic point}). Then, we can consider $g = \pr_{\p^{1,an}}\circ f' : C_V \rightarrow \p^{1,an}_K$.

Let us assume that $g(x,t) = \eta_{0,1}$, then by the Corollaire 3.5.7 of Ducros \cite{DucrosStructureCourbeAnalytique} the fiber $g^{-1}(\eta_{0,1})$ has dimension $0$. Hence, the morphism is quasi-finite on $(x,t)$, so it is open (lemma \ref{morphisme fini sur meme dimension est ouvert}). So, $g(C_V)$ is an open that contains $g(x,t) = \eta_{0,1}$, so for all neighborhood $V$ of $(x,t)$, there exists $(y,z) \in V$ such that $|(f_y)'(z)| > 1$ and $f_y(z) = z$ and $(x,t)$ is not persistently indifferent.

Let us assume that $g(x,t) \neq \eta_{0,1}$, then there exists $U$ a neighborhood of $g(x,t)$ such that $U \subset \overline{D^{an}(0,1)} \backslash D^{an}(0,1)$ and the open $g^{-1}(U)$ is a neighborhood of $(x,t)$ wich makes $(x,t)$ persistently indifferent.  
\end{proof}

In the definition \ref{Definition of persistently indifferent}, we work with $X$ not necessarily of dimension 1.

\begin{lem}\label{Unstably and not persistently indifferent is the same}
Let $K$ be algebraically closed, complete, valued field and $X$ be a $K$-analytic space. Then, for any point $(x,t) \in X\times_K \p^{1,an}_K$, 
\begin{center}
$(x,t)$ is persistently indifferent if and only if $g(x,t) \neq \eta_{0,1}$
\end{center}
where $g$ is defined in the lemma \ref{Unstably and not persistently indifferent is the same in dimension one}.
\end{lem}

\begin{proof}
As in lemma \ref{Unstably and not persistently indifferent is the same in dimension one}, if $g(x,t)\neq \eta_{0,1}, (x,t)$ is persistently indifferent.

Assume that $g(x,t) = \eta_{0,1}$ and let $U$ be a neighborhood of $(x,t)$ in $C_V$. By the Corollaire 3.5.7 of Ducros \cite{DucrosStructureCourbeAnalytique} the fiber $g^{-1}(\eta_{0,1})$ is of dimension strictly smaller than the dimension of $U$, hence there exists $\eta_{a,r} \neq \eta_{0,1} \in g(U)$. 

If $\max(|a|, r) >1$, then $(x,t)$ is not persistently indifferent through $U$ and we are done. Let us now assume the contrary, hence $|a|\leq 1$ and $r<1$. So there exists $(y,z)\in U$ such that $|g(y,z)(T-a)| = r<1$ and so $g^*(T-a) \in \mathcal{O}(U)$ is not constant in 	absolute value. So by the theorem \ref{maximum modulus}, there exists $(l, w) \in U$ with $|g(l,w)(T-a)|>1$ and since $|a|\leq 1, |g(l,w)(T)|>1$. Hence, $(x,t)$ is not persistently indifferent through $U$, so it is not persistently indifferent. 
\end{proof}

We can now state the result about the continuous motion of the Julia set. It is a non-archimedean analogous of Mané-Sad-Sullivan's theorem.

\begin{thm}\label{JuliaSetMovesContinuouslyNonArchi}
Let $k$ be a non-archimedean complete valued field and $X$ a $k$-analytic space with no boundary. Let $f$ be an analytic family of rational maps and $x \in X$. Let us assume that $f_x$ has type I repelling points. Then the following are equivalent :
\begin{itemize}
    \item The Julia set moves continously on a neighborhood of $x$, 
    \item There exists $U \subset X$ open with $x\in U$ such that for all $y\in U$, every point of $f_y$ is attracting, repelling or persistently indifferent. 
\end{itemize}
\end{thm}

\begin{proof}
    The proof is based on the presentation by McMullen (\cite{McMullenComplexDynamicsAndRenormalization}) of Mané-Sad-Sullivan's theorem.
    
    $\implies$ Let us assume that the Julia set depends continously on $\lambda$ on a neighborhood $U$ of $x$.
    Let $y\in U$ and let $z \in \p^{1,an}_{\mathcal{H}(y)}$ be an indifferent type I point of period $n$. 
    
   Since $z \notin J_y$ and the Julia sets moves continuously, one can find a neighborhhod $V$ of $(y,z)$ such that for all $(w, t) \in V, t \notin J_w$. Let us now consider the closed analytic space $C^n_V$ where $C_V^n$ is defined in \ref{Definition of the periodic point}. We may restrict $U$ such that $(y,z) \in C^n_V$ and for all $(w,t) \in C^n_V, |(f^n_w)'(t)|\leq 1$. So we can apply theorem \ref{maximum modulus} to see that for all $(w,t) \in C^n_V, |(f^n_w)'(t)| = 1$ and $z$ is a persistent indifferent point.
   
   $\Longleftarrow$ Let us assume that for all $y$ sufficiently close to $x$, every point of $f_y$ is attracting, repelling or persistently indifferent. Then we claim that there exists a neighborhood $V$ of $x$ such that each cycle of repelling, indifferent, attracting points remain repelling, indifferent, attracting throughout $V$. 
   
  Let $U$ be the neighborhood of $x$ where the hypothesis is verified. We may restrain $U$ such that $f_x= \frac{a_d(x)z^d+ \cdots + a_0(x)}{b_d(x)z^d+ \cdots + b_0(x)}$ with $a_i(x), b_i(x) \in \mathcal{O}(U)$. And we may assume that for all $y \in U, \infty_y \notin J_y$. Let $\mathcal{U} = U\times_k \ana^{1,an}_k$ and let $n \geq 1$. We claim that on each connected component of $C^n_\mathcal{U}$ all the periodic cycle remain repelling, indifferent or attracting.
  
Let $C$ be a connected component of $C^n_\mathcal{U}$ and $(y,t) \in C$. There is a neighborhhod $U_{y,t}$ of $(y,t)$ in $C$ such that the periodic points remain repelling, indifferent or attracting. Since $C = \bigcup_{(y,t)\in C} U_{y,t}$ they remain repelling, indifferent or attracting througout the connected component. 
	   
  Hence $V = U$ satisfies what we want.   
   
   \bigskip
   Let $U$ be a neighborhood of $x$ such that forall $y \in U, f_y$ has type I repelling points. It exists since $f_x$ has at least one type I repelling point and it is an open condition. We may suppose that throughout $U$ each cycle of attracting, repelling or indifferent remain attracting, repelling or indifferent .  
   
Let $(x_i)_{i\in I} \in U^I$ be a family of points with $x$ in the closure of this family. Let us denote $(x, J_x) := \lbrace (x,z) | z \in J_x\rbrace$ and $L_x := \overline{\lbrace (y, z) \in X\times_k \p^{1,an}_k, \exists i\in I, x_i = y~\mathrm{and}~z \in J_y\rbrace} \cap \pr^{-1}_X(x)$. 
   
   To show that $(x,J_x) \subset L_x$, it suffices to show that $R_x = \lbrace (x,z) | z$ is a repelling point of $f_x\rbrace \subset L_x$. Let $z$ be a repelling point of $f_x$. Since the points remain repelling in the connected component of $C^n_{U\times_k\ana^{1,an}_k}$ we may find $(x_i, z_i)$ with $z_i$ a repelling point and $(x_i, z_i) \rightarrow (x,z)$. Hence, $(x,z) \in L_x$. Then, $R_x \subset L_x$.
   
   Let us now show that $L_x \subset J_x$. Since the repelling, indifferent, attracting points remain repelling, indifferent, attracting throughout $U$, every family $(x_i, z_i)$ with $z_i$ a repelling point of $f_{x_i}$ gives a repelling point $z_i^{f_{x_i}}$ of $f_x$ of the same period. It is a point of the connected component of $C^n_\mathcal{U}$ in $\pr^{-1}_X(x)$. 
   
   Since $L_x =  \overline{\lbrace (y, z) \in X\times_k \p^{1,an}_k, \exists i\in I, x_i = y~\mathrm{and}~z~\mathrm{is~repelling} \rbrace} \cap \pr^{-1}_X(x)$, it is enough to verify that if $(x_i, z_i) \rightarrow (x, z)$ with $z_i$ a repelling point of $f_{x_i}$, then $z \in J_x$. For $i$ big enough, each $(x_i, z_i) \in U$ and so it gives a repelling point $z_i^{f_{x_i}}$ of $f_x$. It suffices to verify that $(x,z)$ is a cluster point of $(x, z_i^{f_{x_i}})$.  
   
   Assume that it is not true, then there exists an open $V$ of $(x,z)$ such that for all big enough $i, (x, z_i^{f_{x_i}}) \notin V$. We might reduce $V$ such that there exists $V'$ open with $V \cap V' = \emptyset$ and all $(x, z_i^{f_{x_i}})$ are in $V'$. Then all $(x_i, z_i)$ are in $V'$ for $i$ big enough witch is a contradiction to the fact that $(x_i, z_i) \rightarrow (x, z)$.
\end{proof}

Furthermore, if the rational map $f_x$ has a finite number of attracting points, then we have an other equivalent proposition :

\begin{prop}\label{JuliaSetsMovesContinuouslyIfAttractingBounded}
Let $k$ be a non-archimedean complete valued field and $X$ a $k$-analytic space with no boundary. Let $f$ be an analytic family of rational maps of degree $d \geq 2$. Let $x\in X$ such that $f_x$ has type I repelling points. 

Furthermore, let us assume that $f_x$ has a finite number of attracting points.

Then the following are equivalent :
\begin{enumerate}
    \item The Julia set depends continously on $\lambda$ on a neighborhood of $x$, 
    \item The number of attracting cycles is constant on a neighborhood of $x$,
    \item The maximum period of an attracting cycle of $f_\lambda$ is locally bounded at $x$.
\end{enumerate}
\end{prop}

\begin{proof}
$ 1 \implies 2$ By theorem \ref{JuliaSetMovesContinuouslyNonArchi}, there is a neighborhood $U$ of $x$ such that the attracting, indifferent and repelling cycle periodic remain attracting, indifferent and repelling cycle periodic throughout $U$. Since $f_x$ has a finite number of attracting points, the number of attracting cycles is constant on a neighborhood.

$2 \implies 3$ Let $U$ be a neighborhood of $x$ such that for all $y\in U$, $f_y$ has $l$ attracting cycles. Since each attracting cycle of $x$ gives an attracting cycle of $y$, with period smaller or equal, for $y$ close enough to $x$. So, on a neighborhood $V$ of $x$ the maximum period of an attracting cycle is bounded on $V$. 

$3 \implies 1$ Let us assume that it is bounded by $N-1$. Let $n \geq N, y \in V$ and $(y,t)$ an indifferent periodic point of $f_y$ of period $n$. It suffices to show that $(y,t)$ is persistently indifferent. Let $\mathcal{V}$ a neighborhood of $(y,t)$ and let $\mathcal{C}$ be the connected component of $C^n_\mathcal{V}$ that contains $(y,t)$. Since there is no attracting cycle of period $n$ in $V$, the cycle remains indifferent on $C$, otherwise it would be a contradiction with the theorem \ref{maximum modulus} applied to $((f^{\circ n})')^{-1}$. So, $(y,t)$ is persistently indifferent. 

Let now $n$  be smaller than $N$ and $(y,t)$ an indifferent periodic point of period $n$. If $(y,t)$ is not persistently indifferent, then on any neighborhood of $(y,t)$ they are attracting or repelling point of period dividing $n$. If there was only attracting and indifferent points of period dividing $n$ in one of the neighborhood, it would be a contradiction with the maximum modulus principle (see theorem \ref{maximum modulus}).

Let us assume by contradiction that there exists $(y_i, t_i)$ repelling point of period dividing $n$ with $(y_i, t_i) \rightarrow (y,t)$. Since the Julia set is perfect (see for exemple \cite{BenedettoLivre}, Theorem 8.15), and the repelling points of type I are dense in the Julia set if they exist, which is the case here, there exists $(y_i, t_i^j) \rightarrow (y_i, t_i)$ for all $i$, with $(y_i, t_i^j)$ repelling points of period $l_i^j \geq N$. Since, the repelling points of period $\geq N$ remain repelling, for all $j, (y_i, t_i^j)$ gives a repelling point $(y,t^j)$ and $(y,t)$ will be a cluster point of $(y,t^j)$. Hence, $(y,t) \in J_y$ and it is a repelling point, which is a contradiction. 

So, $(y,t)$ is peristently indifferent. 

%Furthermore, since the number of attracting cycles is constant on a neighborhood of $x$, we may reduce the neighborhood $U$ such that if $y\in U$ and $(y,t)$ is an attracting cycle of period $m$ then, for any point in $C^m_{\mathcal{V}}$ is an attracting cycle, where $\mathcal{V}$ is a neighborhood of $(y,t)$. 
%
%Hence, for any periodic indifferent point of period $\leq N$ of $y\in U$, it cannot become attracting on a neighborhood, and so it remains persistently indifferent.
\end{proof}

\begin{rem}
By theorem \ref{NumberAttractingPointsBounded}, if the residual characteristic of $k$ is $0$ or is strictly smaller than the degree $d$ of the family of rational maps, then for all $x\in X, f_x$ has a finite number of attracting points $(\leq 2d-2)$ and the proposition \ref{JuliaSetsMovesContinuouslyIfAttractingBounded} applies.
\end{rem}

\subsection{Some examples of continuous families}
Let us study some examples.

\begin{lem}\label{z^2+ T has only repelling points}
Let us consider $k$ a non-archimedean, complete, algebraically closed, field with residue characteristic $\neq 2$. Let us denote $X = \ana^{1,an}_k \backslash \overline{D^{an}(0,1)}$. Let $f : X \times_k \p^{1, an}_k \rightarrow X \times_k \p^{1, an}_k$ be the analytic family of rational maps of degree $2$ such that for all $x\in X, f_x : z \mapsto z^2 + T(x)$, where $T(x)$ denotes the image of $T$ the coordinate of $\ana^{1,an}_k$ in the residual field $\mathcal{H}(x)$. 
Then, the Julia sets moves continuously on $X$.
\end{lem}
\begin{proof}

Let $z$ be a periodic point of $f_x$ in the algebraic closure of $\h(f_x)$. Then, it is enough to show that $z$ is not an indifferent point. 

First, let us show that if there exists $d\geq 1$ with $f^d(z) = z$, then $|z| = \sqrt{|T(x)|}$. 

One has that $f^{\circ d}(y)$ is a polynomial in $y$ with $f^{\circ d}(y) = y^{2^d} + a_{2^d -2}y^{2^d -2} + a_{2^d -4}y^{2^d -4} + \cdots + a_0$ where $a_{2k}$ is a polynomial in $T(x)$ with integers coefficient and $|a_{2k}|= |T(x)|^{2^{d-1} - k}$.

And necessarily $|z| = \sqrt{|T(x)|}$.

Furthermore $(f^{\circ d})' = ((f^{\circ d-1})^2 + T(x))' = 2 f^{\circ d-1} (f^{\circ d-1})'$. 

By induction, $(f^{\circ d})' = 2^d f^{\circ d-1}f^{\circ d-2}\cdots f\times \mathrm{Id}$. So $(f^{\circ d})'(z) = 2^d z f(z) \cdots f^{\circ d-1}(z)$ and all $z, f(z), \cdots, f^{\circ d-1}(z)$ are periodic points, so they all have absolute value $\sqrt{|T(x)|}> 1$. Hence, 
\begin{center}
$|(f^{\circ d})'(z)|>1$.
\end{center}

So all periodic points are repelling, except $\infty$ which is an attracting periodic point. Therefore, we can use lemma \ref{If f has no indifferent point, then Julia sets moves continuously} to conclude.

However the continuity does not extend to the point $\eta_{0,1} \in \ana^{1,an}_k$. Indeed, the Julia set of $f_{\eta_{0,1}}$ is only one point.
\end{proof}

\begin{lem}\label{Tz^3 + Z^2 moves continuously}
Let us consider $\C$ trivially valued and let $X = D^{an}(0,1)\backslash \lbrace 0 \rbrace$. So $X = \lbrace \eta_{0,r}, 0<r<1\rbrace$. Let $f : X \times_k \p^{1, an}_k \rightarrow X \times_k \p^{1, an}_k$ be the analytic family of rational maps of degree $3$ such that for all $x\in X, f_x : z \mapsto T(x)z^3 + z^2$. 

Then, the Julia sets moves continuously on $X$. 
\end{lem}
\begin{proof}

First, since $\tilde{f_x} = z^2$, $f_x$ has $\eta_{0,1}$ as a type II repelling point, hence it has infinitely many non-repelling points. The fixed point of $f_x$ are $\lbrace 0, \infty, \frac{-1 + \alpha}{2T}, \frac{-1 - \alpha}{2T}\rbrace$ with $\alpha \in \mathcal{H}(f_x)$ such that $\alpha^2 = 1+ 4T(x)$. Hence, one of $\frac{-1 + \alpha}{2T(x)}, \frac{-1 - \alpha}{2T(x)}$ has absolute value $\frac{1}{|2T(x)|}>1$, and $f'_T( \frac{-1 \pm \alpha}{2T(x)}) = 3 - \frac{-1 \pm \alpha}{2T(x)}$ and so for one of them $|f'_x( \frac{-1 \pm \alpha}{2T(x)})| >1$ and $f_x$ has type I repelling points.

Since $X \subset \ana^{1,an}_\C$, we will only need to show that no periodic points are unstably periodic, by the lemma \ref{Unstably and not persistently indifferent is the same}. 

Let us assume that $(x,t)$ is an unstably periodic point so $\pr_{\p^{1,an}_k} \circ f'(x,t) = \eta_{0,1}$. So we have a morphism of valued field $\mathcal{H}(\eta_{0,1}) \rightarrow \mathcal{H}(x,t)$. Let us write $x = \eta_{0,r}$, with $0 < r < 1$, then $\mathcal{H}(x) = \C((T))$ where $|T| = r$. The field $\mathcal{H}(x,t)$ is an algebraic extension of $\mathcal{H}(x)$ and so it can be isometrically embedded in $\bigcup_{n\in \N^*} \C((T^{\frac{1}{n}}))$. But $\mathcal{H}(\eta_{0,1}) = \C(T_2)$ with $|T_2|=1$, so we can't get a valued morphism $\mathcal{H}(\eta_{0,1}) \rightarrow \mathcal{H}(x,t)$. Otherwise, $T_2$ would be send to a element of $\mathcal{H}(x,t)$ of absolute value 1. So, it is of the form $a + \sum_{\alpha \in \Q_{>0}} a_\alpha T^\alpha,$ with $a, a_\alpha \in \C$. Then, $T_2 -a$ is send to an element of norm $<1$ and $|T_2 -a| =1$.  

Hence, $(x,t)$ is not unstably periodic. So the Julia set moves continuously.  
\end{proof}

We also have some results of continuity, not on a neighborhood but towards a given rational map.

\begin{lem}\label{If f has no indifferent point, then Julia sets moves continuously}
Let $k$ be a non-archimedean, complete, valued field and $X$ a $k$-analytic space. Let $f$ be an analytic family of rational maps of degree $d\geq 2$. Let assume that $char(\tilde{k}) = 0$ or $char(\tilde{k}) = p>d$. Let $x \in X$ such that $f_x$ has repelling points and assume that $f_x$ has a finite number of indifferents point and all of them are persistently indifferent. Then, for any family $f_y \rightarrow f_x$, the Julia set of $f_y$ moves continuously towards the Julia set of $f_x$.
\end{lem}
\begin{proof}

Since all indifferent points are persistently indifferent, there is a neighborhood $U_{(x,t)} \in C^n_X$ for all indifferent points $(x,t)$ such that all points of $U_{(x,t)} \cap C^n_X$ are indifferent. Since $\pr_X : C^n_X \rightarrow X$ has finite fibers, it is open by lemma \ref{morphisme fini sur meme dimension est ouvert} and $\pr_X(U_{x,t})$ is open in $X$. Since there are a finite number of indifferent points, we might consider the neighborhood of $x, U := \bigcap U_{x,t}$.

Let $f_y \rightarrow f_x$. 

As in the proof of the theorem \ref{JuliaSetMovesContinuouslyNonArchi}, $J(f_x)$ is contained in the closure of the Julia set of the $f_y$'s. (Each repelling point of $x$ comes from a repelling point, for $y$ close enough to $x$.)

Let now $(y, t_y)$ be some repelling points converging to a point $(x,z)$. For $y$ close enough to $x, y$ belongs to $U$, hence if we look at the connected component containing $(y,t_y)$ it won't give an indifferent point at the intersection with $\pr^{-1}_X(x)$ and so it will be a repelling point of $x$. We conclude as in the proof of the theorem \ref{JuliaSetMovesContinuouslyNonArchi}.
\end{proof}

\begin{ex}
Let $k$ be a complete, non-archimedean valued field, and let $f$ be a $k$-analytical family of rational map. Let $x \in X$ be such that the classical Julia set of $f$ is non-empty and compact in $\p^{1}(\hat{\overline{k}})$. Then for any family $f_y \rightarrow f_x$, the Julia set of $f_y$ moves continuously towards the Julia set of $f_x$.

We might use proposition \ref{if classical julia set compact, then one attracting only}, to see that $f_x$ has no indifferent points and lemma \ref{If f has no indifferent point, then Julia sets moves continuously} allows us to conclude.
\end{ex}

\subsection{The continuity is independant of the field}\label{subsection continuity independant of the field}

Let $k$ be a complete, non-archimedean valued field and $K$ be a complete valued field extension. Let $X$ be a $k$-analytic space without boundary and $f$ a family of rational maps. We want to see if there is a link between the continuity of the Julia set on $X$ and the continuity of the Julia on $X_K = X \hat{\otimes} K$.

\begin{nota}
Let $k$ be a complete, non-archimedean valued field and $K$ be a complete valued field extension. Let $X$ be a $k$-analytic space without boundary and $f$ a family of rational maps. Let us denote $X_K = X \hat{\otimes} K$ and $f_K$ the family of rational map, $f_K : (X \times_k \p^{1,an}) \hat{\otimes} K \rightarrow (X \times_k \p^{1,an}) \hat{\otimes} K$. Let us denote $\pi^K : (X \times_k \p^{1,an}) \hat{\otimes} K \rightarrow X \times_k \p^{1,an}$ and $\pi^K_X : X_K \rightarrow X$ the two projections.
\end{nota}

\begin{prop}\label{If moves continuously on X, then also on X_K}
Let $k$ be a complete, non-archimedean valued field and $K$ be a complete valued field extension. Let $X$ be a $k$-analytic space without boundary and $f$ a family of rational maps. 

Let $U$ be an open subset of $X$. Let us assume that there exists $x\in U$ such that $f_x$ has type I repelling points and the Julia set moves continuously on $U$. Then, the Julia set moves continuously on $(\pi^K_X)^{-1}(U)$.
\end{prop}  

\begin{proof}
Since for any $y \in (\pi^K_X)^{-1}(x),~y$ has type I repelling points and $y \in (\pi^K_X)^{-1}(U)$, it suffices to show that for any $y\in (\pi^K_X)^{-1}(U)$ and $(y,t)$ an indifferent periodic point, then it is persistently indifferent by theorem \ref{JuliaSetMovesContinuouslyNonArchi}. 

Let $y\in (\pi^K_X)^{-1}(U)$ and $(y,t)$ be an indifferent periodic point of some period $l$. By contradiction, let us assume that there exists $(y_n, t_n) \rightarrow (y,t)$ with $t_n$ repelling points of $f_{K,y_n}$ of period dividing $l$.

Then, $\pi^K(y_n, t_n) \rightarrow \pi^K(y,t)$. Let us write $\pi^K(y,t) = (\pi^K_X(y), \tilde{t})$ with $\tilde{t} \in \p^{1,an}_{\mathcal{H}(\pi^K_X(y))}$. Hence, $\tilde{t}$ is an indifferent periodic point of $f_{\pi^K_X(y)}$ and $\pi^K_X(y) \in U$. So, it is a persistently indifferent periodic point and $\pi^K(y_n,t_n) = (\pi^K_X(y_n), \tilde{t_n})$ cannot be repelling points and so  $(y_n, t_n)$ cannot be repelling either.
\end{proof}

\begin{prop}\label{If moves continuously on X_K, then also on X}
Let $k$ be a complete, non-archimedean valued field and $K$ be a complete valued field extension. Let $X$ be a $k$-analytic space without boundary and $f$ a family of rational maps. 

Let $U$ be an open subset of $X_K$. Let us assume that there exists $x\in U$ such that $f_x$ has type I repelling points and the Julia set moves continuously on $U$. Then, for any $\emptyset \neq V\subset \pi^K_X(U)$ open, the Julia set moves continuously on $V$.
\end{prop}

\begin{proof}
Since the Julia moves continuously on $U$, for any $y\in U, f_y$ has type I repelling points and so does any $y \in (\pi^K_X)(U)$. So, it suffices to show that for any $y\in \pi^K_X(U)$ and $(y,t)$ an indifferent periodic point, then it is persistently indifferent by theorem \ref{JuliaSetMovesContinuouslyNonArchi}. 

Let $y\in \pi^K_X(U)$ and $(y,t)$ be an indifferent periodic point of some period $l$.

By reducing $U$ and keeping only the connected component of an element $\alpha$ of $(\pi^K_X)^{-1}(y)$, we may assume that $U$ is connected and so the type of the periodic point do not change throughout $U$.

Let $\mathcal{V}$ be a neighborhood of $(y,t)$ and let $C^l_\mathcal{V}$ be the analytic space of periodic points of period dividing $l$. By reducing $\mathcal{V}$ me may assume that $C^l_\mathcal{V} \cap \pr^{-1}_X(y) = \lbrace (y,t)\rbrace$. 

We want to show that for any $(w,z) \in C^l_\mathcal{V}, |(f^l_w)'(z)|=1$.

Let $\mathcal{C} := (\pi^K)^{-1}(C^l_\mathcal{V}) \cap U\times_K\p^{1,an}$. One has $\mathcal{C} \cap \pr^{-1}_{X_K}(\alpha) = \lbrace (w,z), w = \alpha$ and $\pi^K(w,z) = (y,t)\rbrace \cap U\times_K\p^{1,an}$ and all these elements are indifferent periodic points. Since the Julia set moves continuously on $U$, this indifferent point are persistently indifferent, then $\forall (w,z) \in \mathcal{C}, |(f^l_w)'(z)|=1$.

Hence, for any $(w,z) \in C^l_\mathcal{V}, |(f^l_w)'(z)|=1$.
\end{proof}

\begin{prop}
Let $k$ be a complete, non-archimedean valued field and $K$ be a complete valued field extension. Let $X$ be a $k$-analytic space without boundary and $f$ a family of rational maps. 

Let $U$ be an open subset of $X$. Let us assume that there exists $x\in U$ such that $f_x$ has type I repelling points. Then, 
\begin{center}
the Julia set moves continuously on $U \iff$ the Julia set moves continuously on $(\pi^K_X)^{-1}(U)$.
\end{center}
\end{prop}

We can also look at some link between the continuity of the Julia set on a neighborhood of $x\in X$ and the continuity of the Julia set on a neighborhood of some $y\in (\pi^K_X)^{-1}(x)$. 

\begin{prop}
Let $k$ be a complete, algebraically closed, non-archimedean valued field and $K$ be a complete valued field extension. Let $X$ be a $k$-analytic space without boundary and $f$ a family of rational maps. Let $\sigma_{K/k}^X : X\rightarrow X_K$ the continuous section from definition \ref{definition universal point and continuous section}. Let $x\in X$ such that $f_x$ has type I repelling points. Then,

\begin{align*}
\mathrm{the~Julia~set~moves~continuously} &\mathrm{~on~a~neighborhood~of~} x\\
&\iff \\
\mathrm{the~Julia~set~moves~continuously} &\mathrm{~on~a~neighborhood~of~} \sigma_{K/k}^X(x).
\end{align*} 
\end{prop}

\begin{proof}
$\implies$ This comes from propostion \ref{If moves continuously on X, then also on X_K}.

$\Longleftarrow$ Let $U$ be a neighborhood of $\sigma_{K/k}^X(x)$ such that the Julia set moves continuously.

Then, $V := (\sigma_{K/k}^X)^{-1}(U)$ is a neighborhood of $x$. By contradiction, let us assume that the Julia set does not move continuously on $V$. Then, there exists $y\in V$ and $(y,t)$ an indifferent periodic point of period $l$ and $(y_n, t_n)$ some repelling points with period dividing $l$ and $(y_n, t_n) \rightarrow(y,t)$.

Then, $\sigma_{K/k} (y_n, t_n) \rightarrow \sigma_{K/k}(y,t)$ with $\sigma_{K/k} : X\times \p^{1,an} \rightarrow (X\times \p^{1,an})\hat{\otimes} K$ is the continuous section. Since the Julia set moves continuously on $U$ and $\sigma_{K/k}(y,t) = (\sigma_{K/k}(y), \tilde{t})$ with $\tilde{t} \in \p^{1,an}_{\mathcal{H}(\sigma_{K/k}(y))}$ is an indifferent point, the $\sigma_{K/k}(y_n,t_n) = (\sigma_{K/k}(y_n), \tilde{t_n})$ cannot be repelling and so neither $(y_n, t_n)$.
\end{proof}

\begin{prop}
Let $k$ be a complete, non-archimedean valued field and $K$ be a complete, valued field extension. And let us assume that the extension $k\subset K$ is algebraic and finite or $K$ is the completion of a normal extension of $K$. Let $X$ be a $k$-analytic space without boundary and $f$ a family of rational maps. Let $x\in X$ such that $f_x$ has type I repelling points. Then the following are equivalent :

\begin{enumerate}
\item The Julia set moves continuously on a neighborhood of $x$,
\item The Julia set moves continuously on a neighborhood of all $y  \in (\pi_{K/k}^X)^{-1}(x)$,
\item The Julia set moves continuously on a neighborhood of one $y  \in (\pi_{K/k}^X)^{-1}(x)$.
\end{enumerate} 
\end{prop}
\begin{proof}
$1 \implies 2$ is proposition \ref{If moves continuously on X, then also on X_K}.

$2\implies 3$ is clear.

$3\implies 1$ Let $y  \in (\pi_{K/k}^X)^{-1}(x)$ such that the Julia set moves continously on one of its neighborhood $U$.

One needs to show that $\pi^K_X$ is open. If the field extension is finite, it is clear. If it is the completion of a normal extension, then it comes from Proposition 1.3.5 and Corollary 1.3.6 of \cite{BerkovichLivre} which states that for any such $K$, and any Banach algebra $A$ over $k$, then :
\begin{center}
$\mathcal{M}(A\hat{\otimes}K) /\mathrm{Gal}(k^{sep}/k) \rightarrow \mathcal{M}(A)$
\end{center}
is a homeomorphism where $k^{sep}$ is the separable closure of $k$ in $K$. The oppeness of the quotient map $\mathcal{M}(A\hat{\otimes}K) \rightarrow \mathcal{M}(A\hat{\otimes}K) /\mathrm{Gal}(k^{sep}/k)$ shows the openness of $\pi^K_X$.

Since $\pi^K_X$ is open, $\pi^K_X(U)$ is a neighborhood of $x$ and by proposition \ref{If moves continuously on X_K, then also on X}, the Julia set moves continuously on it. 
\end{proof}

\begin{prop}
Let $k$ be a complete, non-archimedean valued field and $K$ be a complete valued field extension. Let $X$ be a $k$-analytic space without boundary and $f$ a family of rational maps. Let $x\in X$ such that $f_x$ has type I repelling point. Assume that $\Int((\pi^K_X)^{-1}(x)) \neq \emptyset$. Then, the Julia set moves continuously on $\Int((\pi^K_X)^{-1}(x))$.
\end{prop}

\begin{proof}
Let $y \in \Int((\pi^K_X)^{-1}(x))$ and let $(y,t)$ be an indifferent point of period $l$, we need to show that it is persistently indifferent.

If not, then there exists $y_n \in \Int((\pi^K_X)^{-1}(x))$ and $(y_n,t_n)$ repelling points of $f_{y_n}$ with period dividing $l$ and $(y_n, t_n) \rightarrow (y,t)$.

Then, $\pi^K(y_n, t_n) \rightarrow \pi^K(y,t)$ and $\pi^K(y_n, t_n) = (x, \tilde{t_n}), \pi^K(y,t) = (x, \tilde{t})$ with $\tilde{t_n}, \tilde{t} \in \p^{1,an}_{\mathcal{H}(x)}$. But $(x, \tilde{t_n}) \in J_x$ since they are repelling periodic points of $f_x$ and $(x, \tilde{t}) \notin J_x$ since it is an indifferent periodic point of $f_x$, which is a contradiction to the fact that $J_x$ is closed. 

\end{proof}

\subsection{A notion of continuity on \texorpdfstring{$\M_d^{an}$}{}}\label{subsection continuity on M_d^an}

In this section, we will focus on the motion of Julia sets on $\M_d^{an}$. Let us first define this notion.

\begin{defi}\label{definition of continuous motion in M_d^an}
Let $k$ be a complete valued field and let $\M_d^{an}$ be the analytification of $\M_d$ and let $\pi^{an} : \Rat_d^{an} \rightarrow \M_d^{an}$ the analytification of the projection morphism. Let $U$ be an open subset of $\M_d^{an}$, then we say that the Julia set moves continuously on $U$ if the following condition is verified:
\begin{center}
$\forall [f] \in \mathcal{U}, \forall \tilde{f} \in (\pi^{an})^{-1}([f])$ there exists $\tilde{f} \in U_{\tilde{f}} \subset \Rat_d^{an}$ such that the Julia set moves continuously on $U_{\tilde{f}}$.
\end{center}
\end{defi}

One wishes to connect this definition with the notion of continuous motion on $\Rat_d^{an}$. M. Maculan studied the application $\pi^{an}$ in \cite{MaculanGIT} and he showed the following result.

\begin{thm}\label{Theorem of Marco Maculan}(Theorem 1.6, \cite{MaculanGIT})

Let $k$ be a complete valued field, $X$ a $k$-affine scheme and $G$ a reductive, algebraic group acting on $X$. Then the morphism $\pi^{an} : X^{an} \rightarrow (X//G)^{an}$ satisfies the following properties:
\begin{itemize}
\item $\pi^{an}$ is surjective and $G^{an}$-invariant,
\item for every $x, x' \in X^{an}$,
\begin{center}
$\pi^{an}(x) = \pi^{an}(x')$ if and only if $\overline{G^{an}\cdot x} \cap \overline{G^{an}\cdot x'} \neq \emptyset$,
\end{center}
\item for every $U \subset X^{an}$ open, $G^{an}$-stable, the projection $\pi^{an}(U)$ is open. 
\end{itemize}
\end{thm}

Hence, we want to show that if $f \in \Rat^{an}_d$ has a neighborhood where the Julia set moves continuously, then we may take a bigger neighborhood which will be $\SL_2^{an}$-stable. If this is true, it will become equivalent that the Julia sets moves continuously on a neighborhood of $f$ and on a neighborhood of $\pi^{an}(f)$. 

Let us first show how to construct this $\SL_2$-stable neighborhood. 

\begin{prop}\label{Construction of SL2 invariant open}
Let $k$ be a complete valued field and $X$ be an affine $\SL_2$-invariant subscheme of $\Rat_d$. Let $f \in X^{an}$ such that there exists $f\in U$ such that the Julia set moves continuously on $U$ and all $g$ in $U$ have type I repelling points. 

Let us define 
\begin{center}
$W = \lbrace x \in X^{an}, \exists M \in \SL_{2,\mathcal{H}(x)}^{an}, M\cdot x \in U \rbrace.$
\end{center}
Then, $W$ is an open, $\SL_2$-invariant and the Julia set moves continuously on $W$. 
\end{prop}

\begin{proof}
First, $W$ is clearly $\SL_2$-invariant. 

Let us show that for $x \in W$, we may find $M \in \SL_{2,\mathcal{H}(x)}^{an}(\widehat{\overline{\mathcal{H}(x)}})$ such that $M\cdot x \in U$.  

Let us denote $\Omega : \SL_2^{an}\times X^{an} \rightarrow X^{an}$ the action of $\SL_2^{an}$. Let $x \in W$. There exists $h\in\SL_{2,\mathcal{H}(x)}^{an}$ such that $(h,x) \in \Omega^{-1}(U)$ and $\Omega^{-1}(U) \cap \pr_2^{-1}(x)$ is an open of $\SL_{2,\mathcal{H}(x)}^{an}$. Since $x$ has repelling type I points, $\mathcal{H}(x)$ is not trivially valued and so the rigid points are dense in $\SL_{2,\mathcal{H}(x)}^{an}$. Hence, there exists $M \in \SL_{2,\mathcal{H}(x)}^{an}(\widehat{\overline{\mathcal{H}(x)}}) \cap (\Omega^{-1}(U) \cap \pr_2^{-1}(x)).$

If $W$ is not open, then there exists $x\in W$ and $x_n \notin W \rightarrow x$ since the Berkovich spaces are angelic (see \cite{PoineauAngelique}).
Let $M\in \SL_{2,\mathcal{H}(x)}^{an}(\widehat{\overline{\mathcal{H}(x)}})$ such that $M\cdot x\in U$. 

By proposition 3.27 of \cite{RoyActionGroupe}, there exists $M_n \in \SL_{2,\mathcal{H}(x_n)}^{an}$ such that $(M_n, x_n) \rightarrow (M,x) \in (\SL_2\times X)^{an}$. Hence, $\Omega(M_n, x_n) \rightarrow \Omega(M,x) \in U$. So, for $n$ big enough $\Omega(M_n, x_n) \in U$ and $x_n \in W$.

Hence, $W$ is a $\SL_2$-invariant open subset of $X^{an}$. It remains to show that the Julia set moves continuously on $W$. So we need to show that for all $g \in W$ any indifferent fixed point is persistently indifferent. 

Let $g \in W$, and $t$ an indifferent point of $g$. Since $g \in W$, there exists $M \in \SL_{2,\mathcal{H}(x)}^{an}(\widehat{\overline{\mathcal{H}(x)}})$ such that $g^M \in U$. Then, $M^{-1}\cdot t$ is an indifferent point of $g^M$. 

Let us assume that $(g_n, t_n) \rightarrow (g,t)$ with $t_n$ a periodic point of $g_n$ of the same period as $t$. Then, by proposition 3.27 of \cite{RoyActionGroupe}, there exists $M_n \in \SL_{2,\mathcal{H}(x_n)}^{an}$ such that $(M_n, g_n, t_n) \rightarrow (M, g, t) \in (\SL_2\times X \times \p^1_k)^{an}$. Then $(g_n^{M_n}, M_n^{-1} \cdot t_n) \rightarrow (g^M, M^{-1}\cdot t)$ and $M_n^{-1} \cdot t_n$ is a periodic point of $g_n^{M_n}$ of same period of $M^{-1}\cdot t$. Since $M^{-1}\cdot t$ is persistently indifferent, for $n$ big enough, $M_n^{-1}\cdot t_n$ is indifferent, hence so is $t_n$ and finally $t$ is persistently indifferent. 
\end{proof}

\begin{prop}\label{Julia moves continuously on Rat_d is the same as in M_d}
Let $k$ be a complete valued field and $X$ be an affine $\SL_2$-invariant subscheme of $\Rat_d$. Let $f \in X^{an}$ such that $f$ has repelling type I points and let $\pi^{an} : X^{an} \subset \Rat_d^{an}\rightarrow (X/\SL_2)^{an}\subset \M_d^{an}$. Then,
\begin{align*}
\mathrm{There~is~a~neighborhood~of~}f \in &X^{an} \mathrm{~such~that~the~Julia~set~moves~continuously}\\
&\iff\\ 
\mathrm{There~is~a~neighborhood~of~}\pi^{an}(f) \in& (X/\SL_2)^{an} \mathrm{~such~that~the~Julia~set~moves~continuously.} 
\end{align*}
\end{prop}

\begin{rem}
Here we only consider familes of rational maps that are defined by algebraic families because it is the setting where the construction of the $M_n$ are defined in \cite{RoyActionGroupe}. Nevertheless smoooth analytic spaces are locally algebraic (see for example, 0.21 of \cite{DucrosModelTheory} which relies on Theorem 3.4.1 of \cite{BerkovichEtaleCohomology}) and so, the result will stand if we have a smooth analytic family of rational maps. 
\end{rem}

\begin{proof}
$\Longleftarrow$ It is clear by definition.

$\Longrightarrow$ Let $f\in U \subset X^{an}$ such that the Julia set moves continuously on $U$. Then, we can define $W$ as in proposition \ref{Construction of SL2 invariant open} and by theorem \ref{Theorem of Marco Maculan} (of Maculan \cite{MaculanGIT}), $\pi^{an}(W)$ is open and contains $\pi^{an}(f)$. 
\end{proof}

\section{Continuity of the Julia set in the hybrid setting}\label{Section continuity in the hybrid settings}
In this section, one will consider $\Rat_d$ as a scheme defined over $\C$. We will work over $\C_{hyb}$. 

As in non-archimedean setting, we define a notion of analytical family of rational map.

\begin{defi}
Let $X$ a $\C_{hyb}$-analytic space. An analytic family of rational maps of degree $d\geq 2$, parametrized by $X$, is an analytic map $f : X \times_{\C_{hyb}} \p^{1,an}_{\C_{hyb}} \rightarrow X\times_{\C_{hyb}} \p^{1,an}_{\C_{hyb}}$ such that $\pr_X = \pr_X \circ f$ where $\pr_X : X\times_{\C_{hyb}} \p^{1,an}_{\C_{hyb}}$ is the projection and for all $x \in X, f_x := f(x, \cdot) : \p^{1,an}_{\mathcal{H}(x)} \rightarrow \p^{1,an}_{\mathcal{H}(x)}$ is a rational map of degree $d\geq 2$.
\end{defi}

Firstly, we present some special cases where we have continuity, not necessarily on a neighborhood but towards a point. 

\begin{prop}\label{if no indifferent point, then julia set moves continuously in the hybrid setting}
Let $f \in (\Rat_d)_0$ be a rational function. Let assume that $f$ has no indifferent points. Then the Julia set moves continuously at $f$, in the sense that if $f_i \rightarrow f$, then 
\begin{center}
$\mathrm{Closure}(\lbrace (f_i, z_i), z_i \in J_{f_i}\rbrace) \cap \pr^{-1}(f) = \lbrace (f,x), x\in J_f\rbrace$ 
\end{center}
with $\pr : \Rat_d\times \p^{1,an} \rightarrow \p^{1,an}$.
\end{prop}

\begin{proof}
By theorem \ref{NumberAttractingPointsBounded}, we see that $f$ has type I repelling points and by proposition \ref{2 attracting points implies infinitely many}, we see that $f$ has exactly one non-repelling point, which is attracting. Furthermore, by proposition \ref{repelling points of type I are dense in the Julia Set}, the type I repelling points are dense in the Julia set. 

By lemma \ref{If f has no indifferent point, then Julia sets moves continuously}, one sees that if we have a family $f_i \in (\Rat_d)_0 \rightarrow f$ then, the Julia set of $f_i$ moves continuously towards the Julia set of $f$.

Let $f_i \in (\Rat_d)_{>0} \rightarrow f$. 

Let $z$ be a repelling point of $f$ and let $V$ be a neighborhood of $(f,z)$ small and $C^n_V$, the closed analytic space defined in the same way as in definition \ref{Definition of the periodic point}. By reducing $V$ we might consider that in the entire connected component of $C^n_V$ that contains $(f,z)$ the points are repelling. Hence, for all $i$ big enough, we may consider $z_i$ a repelling point of $f_i$ such that $(f_i, z_i) \rightarrow (f,z)$.

Hence, we have $\lbrace (f,x), x\in J_f\rbrace \subset \mathrm{Closure}(\lbrace (f_i, z_i), z_i \in J_{f_i}\rbrace) \cap \pr^{-1}(f)$.

Let us now consider $(f_i, z_i)$ with $z_i$ a repelling point of $f_i$. We can assume that $(f_i, z_i)$ converges to some $(f,x)$. We need to show that $x \in J_f$. For each $z_i$ let us denote $n_i$ the period of the cycle and we might look at $C^{n_i}_{U\times \ana^{1,an}}$ where $U$ is a neighborhood of $f$ that contains $f_i$ for $i$ big enough. Then, for each $z_i$ we might look at the connected component where $(f_i, z_i)$ belongs. It crosses $\pr^{-1}(f)$ at some point $(f, z_i^f)$ which is necessarily repelling since all the periodic points of $f$ are repelling.

Then $(f,x)$ is a cluster point of the $(f, z_i^f)$. If not we may find $V, V'$ two open subsets such that $(f,x) \in V, (f, z_i^f)\in V'$ and $V \cap V' = \emptyset$, but for $i$ big enough, it would imply that $(f_i, z_i) \in V'$ which is a contradiction to the fact that $(f_i, z_i) \rightarrow (f,x)$. 

Hence, we have $\mathrm{Closure}(\lbrace (f_i, z_i), z_i \in J_{f_i}\rbrace) \cap \pr^{-1}(f) \subset \lbrace (f,x), x\in J_f\rbrace$.
\end{proof}

\begin{rem}\label{If infinitely many non repelling then no continuous motion}
The condition of the previous proposition implies that there are a finite number of non-repellings points which is a necessary condition to have continuity of the Julia set.

Indeed, since the number of non-repelling points of $f\in \Rat_d(\C)$ is bounded by $2d-2$, there can be at most $2d-2$ non-repelling periodic point for $f\in Rat_{d,0}$. Hence, if the Julia set moves continuously at $f\in \Rat_{d,0}$, either $f$ has one attracting point and all the others are repelling or there is $\leq 2d-2$ indifferent point and no attractiong point. This comes from propositionn \ref{2 attracting points implies infinitely many} which states that if they are an attracting point and an other non repelling point, then they are infinitely many non repelling points in the non-archimedean case.
\end{rem}

\begin{ex}
Let $f(z) = z^2 + c \in (\Rat_2)_0$ with $|c| > 1$. Then the Julia set moves continuously at $f$.
\end{ex}

\begin{proof}
As seen in lemma \ref{z^2+ T has only repelling points}, $f$ has no indifferent point. Therefore, we can use proposition \ref{if no indifferent point, then julia set moves continuously in the hybrid setting} to conclude.
\end{proof}

\subsection{The polynomial case}

Let us focus on polynomials and the filled Julia set. In the same way that the Julia set is "lower semicontinuous", we have that the filled Julia set is "upper semicontinuous". Here, lower semicontinous (resp. upper semicontinuous") means that the limit of the Julia sets is contained in (resp. contains) the limit Julia set. If we work with a countable subfield of $\C$ such as $\Q[i]$, then $\Rat_{d,\Q[i]}^{hyb}$ is metrizable and this notion corresponds to lower semicontinuity with respect to the Haussdorf topology. This part has been inspired by the article of A. Douady, \cite{DouadySemiContinuite}.

Firstly, let us recall that the Julia set is lower semicontinuous in general.
\begin{lem}\label{JuliaSetSemiContinuous}
Let $k$ be a valued field. Let $X$ be a $k$-analytic space (in this case, $k$ is suppose to be complete) or a $k_{hyb}$-flowing, analytic space. Let $f : X\times \p^1 \rightarrow X\times \p^1$ an analytic family of rational maps of degree $d$. 

Let $f_{x_i} \in X \rightarrow f_x\in X$, then 
\begin{center}
$\lbrace (f_x, z), z\in J(f_x)\rbrace \subset \mathrm{Closure}(\lbrace (f_{x_i}, z), z \in J(f_{x_i})\rbrace) \cap \pr_X^{-1}(f_x)$.
\end{center}
\end{lem}

\begin{proof}
This is a direct consequence of Poineau's theorem (see Théorème 6.26 \cite{PoineauDynamique}). This theorem states that the family of equilibrium measure associated to each $f_y$ is continuous and since the Julia set of $f_y$ is the support of the equilibrium measure of $f_y$, the lower semicontinuity follows.
\end{proof}

Let us now recall the definition of the filled Julia set.
\begin{defi}\label{def filled julia set, escape radius}
Let $k$ be a complete valued field and $f$ be a polynomial of degree $d \geq 2$ defined over $k$. Then one defines the filled Julia set as a closed subset of $\p^{1,an}_k$ in the following manner :
\begin{center}
$K_f := \lbrace x \in \p^{1,an}_k, f^{\circ n}(x) \nrightarrow \infty \rbrace.$
\end{center}

We can see $\mathcal{P}_d$, the set of polynomials of degree $d$ as a closed subscheme of $\Rat_d$.

For each $f \in \mathcal{P}_d$, one can define the escape radius 
\begin{center}
$R_f := \frac{1 + |a_d| + \cdots +|a_0|}{|a_d|} = \eta_f(\frac{b_0^{2d}}{\rho})^{\frac{1}{d}} + \eta_f(\frac{a_d^d b_0^d}{\rho})^{\frac{1}{d}}+ \cdots + \eta_f(\frac{a_0^d b_0^d}{\rho})^{\frac{1}{d}}$.
\end{center}
It defines a continuous function $\mathcal{P}_d \rightarrow \R_+$. Furthermore, $K_f = \lbrace x \in \p^{1,an}, \forall n ~ |T(f^{\circ n}(x))|\leq R_f\rbrace$, where we see $f^{\circ n}(x)$ as an element of $\p^{1,an}$.
\end{defi}

\begin{prop}(Theorem 8.15, (f), of \cite{BenedettoLivre})

Let $k$ be a complete valued field and $f$ be a polynomial of degree $d \geq 2$ defined over $k$. Then, 
\begin{center}
$J_f = \delta K_f$.
\end{center}
\end{prop}

\begin{prop}
Let $k$ be a complete valued field or a hybrid field, then the filled Julia set is upper semicontinuous on $P^{an}_d$. 
\end{prop}

\begin{proof}
Let $f \in \mathcal{P}_d^{an}$ and let $f_i \rightarrow f$. Let $z\in \p^{1,an}_{\h(f)}$ such that there exists $z_i \in K_{x_i}$ and $(f_i, z_i) \rightarrow (f,z)$. Let $U$ be a neighborhood of $(f,z)$ such that the coefficients $a_i, b_i$ of $f$ are elements of $\mathcal{O}(U)$, then $F^{\circ n} : (f,z) \mapsto f^{\circ n}(z) \in \mathcal{H}(f,z)$ is an element of $\mathcal{O}(U)$ and for $i$ big enough, $(f_i, z_i) \in U$. 

Then, $|F^{\circ n}(f_i, z_i)| \rightarrow |F^{\circ n}(f, z)|$ and $|F^{\circ n}(f_i, z_i)|\leq R_{f_i}$. Since $R_{f_i} \rightarrow R_f$, one gets $z \in K_f$. 
\end{proof}

\begin{cor}
Let $f \in (\mathcal{P}_d)_0$ be a polynomial of degree $d\geq 2$ such that $K_f$ has empty interior, then the Julia set moves continuously at $f$.
\end{cor}

\begin{rem}
The case where $K_f$ has empty interior is the case where $f$ has only one non-repelling point which is attracting. Furthermore, in this case the classical Julia set is non-empty and compact, by Trucco \cite{TruccoWanderingFatouComp} if $K$ is the algebraic closure of a discretely valued field.  

In fact, Trucco showed that if $f \in (P_d^{an})_0$, then the Julia set is the disjoint union of the grand orbit of a finite number of repelling type II points and the classical Julia set (Theorem A, \cite{TruccoWanderingFatouComp}). If there is a type II repelling point, then there are infinitely many non-repelling points, hence there cannot be a continuous motion of the Julia sets, see remark \ref{If infinitely many non repelling then no continuous motion}.

Furthermore, he proved that in the polynomial case, the fact that the classical Julia set is non-empty and compact is equivalent to the fact that all the non-infinity periodic points are repelling (Corollary 7.15, \cite{TruccoWanderingFatouComp}). Hence, this implies the following result. 
\end{rem}

\begin{prop}
Let $f \in (\mathcal{P}_d)_0$ be a polynomial of degree $d \geq 2$. Then, the Julia sets moves continuously at $f$ if and only if $f$ has only repelling points and $\infty$ as an attracting point. 
Furthermore if $\mathcal{H}(f)$ is discretely valued, then it is also equivalent to the classical Julia set of $f$ being compact in $\p^1(\hat{\overline{\mathcal{H}(f)}})$.
\end{prop}

\begin{proof}
The second if and only if comes directly from the corollary 7.15 of Trucco, \cite{TruccoWanderingFatouComp}, and lemma \ref{compact of k are compact in berkovich}. 

The first implication comes from theorem A of Trucco, \cite{TruccoWanderingFatouComp}, since it shows that the classical Julia set is not compact if and only if $f$ has a type II repelling point which implies that it has infinitely many non-repelling points.

The converse implication is proposition \ref{if no indifferent point, then julia set moves continuously in the hybrid setting}. 
\end{proof}

\subsection{A Continuous motion in the hybrid setting}

As in the non-archimedean case, we will try to find some conditions to ensure that if $f$ is a family of rational maps then, the Julia sets move continuously on some open subset. 

If $f_x$ is a element such that $x\in X_{>0}$, then we might find an open subset $U \subset X_{>0}$ and the continuity on a neighborhood of $X$ and the Julia sets will move continuously if they are stable in the sense of Mané-Sad-Sullivan, \cite{MSS}. 

Since we can study the continuity in the archimedean setting and in the non-archimedean setting, we could think that the Julia sets will move continuously if there is a neighborhood $U$ of $x$ such that the Julia set moves continuously on $U \cap X_{>0}$ and on $X \cap X_0$. Unfortunately, it is not true. 

\begin{ex}\label{continuity in hybrid is not continuity on U_0 and U_>0}
Let $t \in \C$ with $|t|<1$. Let $f_t(z) = tz^3 + z^2$, then the critical point of $f_t$ are $z_t = 0$ and $z_t = -\frac{2}{3t}$, hence for $|t|<1$ the critical points are never in the Julia set. So the Julia moves continuously, in the archimedean case, by the caracterisation of the continuity of the Julia set with the crititcal points (see for example Theorem 4.2, 7. of \cite{McMullenComplexDynamicsAndRenormalization}). 

We saw in lemma \ref{Tz^3 + Z^2 moves continuously} that the Julia set moves continuously on $D^{an}(0,1)_{(\C, |\cdot|_0)}$. But if we take $f_T \in (Rat_d)_0$ it has infinitely many non-repelling points, so the Julia set cannot move continuously in the hybrid case since in the archimedean case, there is only finitely many non-repelling points.  
\end{ex}

We will focus on the continuity motion on a flowing neighborhood since we want to apply it on the compactification $\Rat_d^\urcorner$, as defined by Poineau \cite{PoineauCompactificationHybride}, of $\Rat_d(\C)$. But this condition is superfluous.

\begin{lem}\label{Julia set is invariant through the flow}
Let $\C$ be equipped with the hybrid norm and $X$ a flowing $\C_{hyb}$-analytic space. Let $f$ be an analytic family of rational maps. Then, the Julia set is inviariant through the flow.

In the sense that if $(x,z) \in J_x$ with $J_x$ embedded in $X\times \p^{1,hyb}$, then for any $\alpha \in (0, \frac{1}{\pr (x,z)}, (x,z)^\alpha \in J_{x^\alpha}$.
\end{lem}

\begin{proof}
Let $(x,z) \in J_x$. Since the repelling points are dense in $J_x$ (proposition \ref{Repelling points are dense}), it is enough to show that $(x,z)^\alpha \in J_{x^\alpha}$ with $z$ a repelling point.

If $z$ is a type I repelling point, then $z^{\alpha}$ is also a type I repelling point of $x^\alpha$ since $|f'_{x^\alpha}(z^\alpha)| = |f'_x(z)|^\alpha$ and $(x,z)^\alpha \in J_{x^\alpha}$.

If $z$ is a type II repelling point, then there exists $M \in \SL_2(\overline{\mathcal{H}(x)})$ such that $M\cdot z = \eta_{0,1}$ and $f_x^M$ has a reduction in $\tilde{\mathcal{H}(x)}$ of degree at least 2.

Since, $M^\alpha \cdot z^\alpha = (M\cdot z)^\alpha$ and $(f_x^\alpha)^{M^\alpha} = (f_x^M)^\alpha$ and the reduction of $f^M_x$ and $(f_x^\alpha)^{M^\alpha}$ have same degree, $z^\alpha$ is a type II repelling point of $f_x^\alpha$ and $(x,z)^\alpha \in J_{x^\alpha}$.
\end{proof}

\begin{lem}\label{Julia set moves continously on neighborhood is the same as on a flowing neighborhood}
Let $\C$ be equipped with the hybrid norm and $X$ a flowing $\C_{hyb}$-analytic space. Let $f$ be an analytic family of rational maps. Let $x\in X$ and $U$ a neighborhood of $x$ such that the Julia set moves continuously on $U$. 

Then, the Julia set moves continuously on $T(U)$, the trajectory of $U$ which is an open flowing neighborhood.
\end{lem}

\begin{proof}
Let $y \in T(U)$ and $y_i \rightarrow y$. Since $y \in T(U)$, there exists $\alpha \in (0, \frac{1}{\pr (y)}]$ such that $y^\alpha \in U$. 

If $\pr(y) \neq 0$ and $\alpha = \frac{1}{\pr (y)}$, then we define $\alpha_i = \frac{1}{\pr (y_i)}$. If not then for $i$ big enough, $\alpha < \frac{1}{\pr (y_i)}$ and we define $\alpha_i = \alpha$.

Either way, we have $y_i^{\alpha_i} \rightarrow y^\alpha \in U$.

Let $z_i \in J_{y_i}$ and let $(y,z)$ a cluster point of $(y_i, z_i)$. We want to show that $z \in J_y$. Let us assume that $(y_i, z_i) \rightarrow (y,z)$. 

Since $(y_i, z_i)^{\alpha_i} \rightarrow (y,z)^\alpha$ and the Julia set moves continuously on $U, (y,z)^\alpha$ is a point of the embedding of $J_{y^\alpha}$ in $X\times\p^{1,hyb}$. So, $z\in J_y$.
\end{proof}

As in non-archimedean case, we will need a notion of persistently indifferent point but in the hybrid setting.
First let us state a lemma. 

\begin{lem}
Let $\C$ be equipped with the hybrid norm and $X$ a flowing $\C_{hyb}$-analytic space. Let $f$ be an analytic family of rational maps. Then, the closed analytic space of periodic point of period $n, C^n_X$ is flowing.
\end{lem}

\begin{proof}
It is clear by definition of $C^n_X$.
\end{proof}

\begin{defi}
Let $\C$ be equipped with the hybrid norm and $X$ a flowing $\C_{hyb}$-analytic space. Let $f$ be an analytic family of rational maps and $x \in X_0$. A periodic point $z \in \p^{1,an}_{\mathcal{H}(x)}$ of $f_x$ of period $n$ of type I is said to be hybridly persistently non reppelling if it is persistently indifferent on $X_0$ and there is a flowing neighborhood $U$ of $(x,z)$ such that if $V$ be the connected component of $C^n_U$ containing $(x,z)$ and $V^1, \cdots, V^m$ are the connected components of $V_{>0}$, then for all $i$ we are in one of the following cases : 
\begin{itemize}
    \item for all $(y,t) \in V^i,~t$ is an attracting point of $f_y$, 
    \item for all $(y,t) \in V^i,~t$ is an indifferent point of $f_y$ not in the Julia set $J_{f_y}$. 
\end{itemize}
\end{defi}

\begin{rem}
This definition is really a local definition : if $U$ is the neighborhood in the definition, then any other flowing neighborhood $U'\subset U$ will satisfy the conditions of the definition. 
\end{rem}

We can now state the result about the continuous motion in the hybrid case. 

\begin{thm}\label{Julia set moves continuously hybrid setting}
Let $\C$ be equipped with the hybrid norm and $X$ a flowing $\C_{hyb}$-analytic space. Let $f$ be an analytic family of rational maps and $x \in X$ such that $x$ has type I repelling points. 

\begin{itemize}
\item If $x \in X_{>0}$, then the Julia set moves continuously on a neighborhood of $x$ if and only if it moves continuously on a neighborhood $U$ of $x$ in the fiber $X_{\pr(x)}$,

\item If $x\in X_0$, then the Julia set moves continuously on a neighborhood of $x$ if and only if there is a flowing neighborhood $U$ of $x$ such that :
\begin{enumerate}
\item the Julia sets moves continuously on $U_{>0}$,
\item for all $y\in U_0$, every indifferent point of $f_y$ is hybridly persistently non-repelling.
\end{enumerate} 
\end{itemize} 
\end{thm}

\begin{rem}
In the non-archimedean case, one has to assume that $X$ does not have a boundary, but it is not needed here. The reason is that the definition differs on the case of Berkovich spaces over a complete field and in the general case. In the general case one only considers spaces without boundary.

For instance, the closed ball center on $0$ of radius $1$ in $\ana^{1,an}_k$ is an analytic space but the closed ball center on $0$ of radius $1$ in $\ana^{1,hyb}_\C$ is not.  
\end{rem}

\begin{proof}
By lemma \ref{Julia set moves continously on neighborhood is the same as on a flowing neighborhood}, we can work with flowing neighborhood.

First let us assume that $x\in X_{> 0}$. If $U$ is a flowing neighborhood of $x$, then so is $U_{>0}$. Hence, we may assume that $U$ is of the form $U_{>0}$.

$\implies$ If the Julia set moves continuously on $U$ then it moves continuously on $U_{\pr(x)} = U \cap X_{\pr(x)}$.

$\Longleftarrow$ Let us assume that the Julia set moves continuously on a neighborhood $U_{\pr(x)}$ of $x$ in $X_{\pr(x)})$. Then, $U := T(U_{\pr(x)})$ is a flowing open subset. 

Let $f\in U$ and let $f_i \in U \rightarrow f$. Let us denote $\varepsilon = \pr(f), \varepsilon_i = \pr(f_i)$. 

One needs to check that 
\begin{center}
$\mathrm{Closure}(\lbrace (f_i, z_i), z_i \in J_{f_i}\rbrace) \cap \pr^{-1}_X(f) = \lbrace (f, z), z\in J_f\rbrace$. 
\end{center}

Since the repelling points remain repelling locally, one has $\lbrace (f, z), z\in J_f\rbrace \subset \mathrm{Closure}(\lbrace (f_i, z_i), z_i \in J_{f_i}\rbrace) \cap \pr^{-1}_X(f)$.

Let $(f_i, z_i) \rightarrow (f,z)$ with $z_i \in J_{f_i}$. One needs to check that $z \in J_f$. We have that $(f_i, z_i)^{\frac{\pr(x)}{\varepsilon_i}} \rightarrow (f,z)^{\frac{\pr(x)}{\varepsilon}}$ and since the Julia set moves continuously on $U_{\pr(x)}$, one has $z^{\frac{\pr(x)}{\varepsilon}} \in J_{f^{\frac{\pr(x)}{\varepsilon}}}$. Hence, $z\in J_f$ by lemma \ref{Julia set is invariant through the flow}.

Now let us assume that $x\in X_0$. 

$\implies$ Let us assume that the Julia set moves continuously on a flowing neighborhood $U$ of $x$. Then it moves continuously on the open $U_{>0}$ and the open $U_0$ of $X_0$. 

Let $y \in U_0$. Since the Julia set moves continuously on $U_0$, all indifferent point of $y$ are persistently indifferent on $U_0$. Let $(y,t)$ be an indifferent point of $y$ of period $n$. Let $V$ be a flowing neighborhood of $(y,t)$ and let $C = C^n_V$. It is a flowing analytic space. Let $V^i$ be a connected component of $C_{>0}$ and $0 < \varepsilon \leq 1$, then $V_\varepsilon^i$ is a connected component of $C_\varepsilon$. Since $t \notin J_y$ we may suppose that for all $(f,z) \in V^i, z\notin J_f$. So for all, $(f,z) \in V^i_\varepsilon, |(f^n)'(z)|\leq 1$. We have two separate cases, either it atteins 1 or not. If it atteins 1, by the maximum modulus principle, it is contant to the value 1 on all $V^i_\varepsilon$, otherwise it is always $<1$. Since $V^i$ is flowing, this behavior remains true in all $V^i$ : either for all points $(f,z) \in V^i, |(f^n)'(z)| <1$ or for all points $(f,z) \in V^i, |(f^n)'(z)| = 1$. And $(y,t)$ is hybridly persistently non-repelling. 

$\Longleftarrow$ Let us assume that there exists a flowing neighborhood $U$ of $x$ such that the Julia set moves continuously on $U_{>0}$ and for all $y\in U_0$, every indifferent point of $y$ is hybridly persistently non-repelling. 

Let us show that the Julia set moves continuously on $U$. 
Let $f\in U$. If $f\in U_{>0}$ since the Julia set moves continuously on $U_{>0}$, we have that for all $f_i \rightarrow f,$ the $\mathrm{Closure}(\lbrace (f_i, z_i), z_i \in J_{f_i}\rbrace) \cap \pr^{-1}_X(f) = \lbrace (f, z), z\in J_f\rbrace$.     

Let us now assume that $f\in U_0$. Let $f_i \rightarrow f$. Let $(f_i, z_i) \rightarrow (f,z)$ with $z_i$ repelling point. One needs to show that $z\in J_f$. Let us look at the two subnets $\lbrace f_i, f_i \in U_0\rbrace, \lbrace f_i, f_i \in U_{>0}\rbrace$. Since at least one of them has an infinite number of elements, we may suppose that the net is of one of this form to show that $z$ has to be in $J_f$.  

Since all its indifferent points are persistently indifferent in $U_0$, if $f_i \in U_0 \rightarrow f$, the Julia set of $f_i$ moves continuously to $J_f$ by theorem \ref{JuliaSetMovesContinuouslyNonArchi} and $z\in J_f$. 

So, let $f_i \in U_{>0} \rightarrow f$. As in the previous case, one needs to show :  
\begin{center}
$\mathrm{Closure}(\lbrace (f_i, t_i), t_i \in J_{f_i}\rbrace) \cap \pr^{-1}_X(f) = \lbrace (f, t), t\in J_f\rbrace$. 
\end{center}

Since the repelling points remain repelling locally, one has $\lbrace (f, t), t\in J_f\rbrace \subset \mathrm{Closure}(\lbrace (f_i, t_i), t_i \in J_{f_i}\rbrace) \cap \pr^{-1}_X(f)$.

Each $z_i$ is in a connected component $V^i$ of some $(C^{n_{z_i}}_V)_{>0}$ where $V$ is a flowing open subset containing $(f_i, z_i)$ such that $\pr_X(V) = U$. Let $z^f_i$ be a point of $(C^{n_{z_i}}_V) \cap \pr^{-1}_X(f)$ such that it is a cluster point of $V^i$. Since $z_i$ is repelling, $z^f_i$ cannot be indifferent, since the indifferent point of $f$ are persistently hybridly non-repelling. Hence, they are repelling. 

As in the proof of theorem \ref{JuliaSetMovesContinuouslyNonArchi}, $(f,z)$ is a cluster point of $(f_i, z^f_i)$. 

So, $z\in J_f$ and  $\mathrm{Closure}(\lbrace (f_i, z_i), z_i \in J_{f_i}\rbrace) \cap \pr^{-1}_X(f) \subset \lbrace (f, z), z\in J_f\rbrace$.

Finally, we conclude that the Julia set moves continuously on $U$.
\end{proof}

\begin{rem}
We have here a difference with the archimedean and non-archimedean case : an indifferent periodic point of a point $f_x \in X_0$ can come from an attracting point of some $f_{x_i} \in X_{>0} \rightarrow f_x$. This cannot happen in the other cases : the type of the attracting point never changes throughout a neighborhood.

But we do not have an example where such a case occurs. 
\end{rem}

\subsection{A notion of continuity on \texorpdfstring{$\M_d^{hyb}$}{}}

In this section we aim to have a similar result as in the subsection \ref{subsection continuity on M_d^an}. First let us give the definition of having a continuous motion of the Julia set in $\M_d^{hyb}$. We will define it in the same way as in definition \ref{definition of continuous motion in M_d^an}.

\begin{defi}\label{definition of continuous motion in M_d^hyb}
Let $\C$ be equipped with the hybrid norm and let $\M_d^{hyb}$ be the analytification of $\M_d$. Let $U$ be a open, flowing subset of $\M_d^{hyb}$. We say that the Julia set moves continuously on $U$ if the following condition is verified:
\begin{center}
$\forall [f] \in U, \forall \tilde{f} \in (\pi^{hyb})^{-1}([f])$ there exists a neighborhood $U_{\tilde{f}}$ of $\tilde{f}$ in $\Rat_d^{hyb}$ such that the Julia set moves continuously on $U_{\tilde{f}}$, 
\end{center}
where $\pi^{hyb} : \Rat_d^{hyb} \rightarrow \M_d^{hyb}$ is the analytification of the projection morphism.
\end{defi}

As in the subsection \ref{subsection continuity on M_d^an}, we would like to have a correspondence between the notion of continuous motion on $\Rat_d^{hyb}$ and the notion of continuous motion on $\M_d^{hyb}$. 

First let us show that, in this setting, if the Julia set moves continuously on a neighborhood of $f$ then, we can suppose that this neighborhood is $\SL_2$-invariant. This corresponds to the proposition \ref{Construction of SL2 invariant open}. 

\begin{prop}\label{Construction of SL2 invariant open in hybrid setting}

Let $\C$ be equipped with the hybrid norm and $X$ be an affine $\SL_2$-invariant subscheme of $\Rat_d$. Let $f \in X^{hyb}$. Assume that there exists a neighborhood $U$ of $f$ such that the Julia set moves continuously on $U$ and $U$ is flowing. 

Let us define 
\begin{center}
$W = \lbrace x \in X^{hyb}, \exists M \in \SL_{2,\mathcal{H}(x)}^{an}, M\cdot x \in U \rbrace.$
\end{center}
Then, $W$ is an open, $\SL_2$-invariant, flowing and the Julia set moves continuously on $W$.
\end{prop}

\begin{proof}
First let us assume that $f\in X_{>0}$. Then, we may assume that $U \subset X_{>0}$. Then, for all $M\in \SL_2(\C), M\cdot U$ is an open that contains $M\cdot f$ and the Julia set moves continuously on this open. Furthermore, $W = \bigcup_{M\in \SL_2(\C)} M\cdot U$, hence it is open, $\SL_2$-invariant and the Julia set moves continuously on $W$.

If $f \in X_0$, the proof that $W$ is open and that the Julia set moves continuously on $W$ is the same as in the proof of proposition \ref{Construction of SL2 invariant open}. The only thing we need to check is that if $x\in W_0$ then $\mathcal{H}(x)$ is not trivially valued and one needs to consider generalised sequences and not sequences as we do not know if the hybrid spaces are angelic. 

Let $x \in W_0$ and let $M\in \SL_{2,\mathcal{H}(x)}^{an}$ such that $M\cdot x \in U$. Hence, $M\cdot x$ does not have potential good reduction : it has repelling type I points since the Julia moves continuously on $U$. Since $M$ is not necessarily a rigid point of $\SL_{2,\mathcal{H}(x)}^{an}$, one needs propositon 4.14 of \cite{RoyActionGroupe} to conclude that $x$ does not have potential good reduction either and so $\mathcal{H}(x)$ cannot be trivially valued. 

We are left to show that $W$ is flowing. Let $x \in W$ and $\alpha \in (0, \frac{1}{\pr(x)}]$, where $\pr : X^{hyb} \rightarrow \mathcal{M}(\C_{hyb})$. Let $M \in \SL_{2, \mathcal{H}(x)}^{an}(\widehat{\overline{\mathcal{H}(x)}})$ such that $M\cdot x \in U$. Then, $(M\cdot x)^\alpha = M^\alpha \cdot x^\alpha$ and $M^\alpha \in \SL_{2, \mathcal{H}(x^\alpha)}^{an}(\widehat{\overline{\mathcal{H}(x^\alpha)}})$. 

Since, $U$ is flowing, $(M\cdot x)^\alpha \in U$ and so $x^\alpha \in W$ and $W$ is flowing. 
\end{proof}

Unfortunately, there is no result about the openess of the morphism $\pi^{hyb} : \Rat_d^{hyb} \rightarrow \M_d^{hyb}$ which would be the counterpart of Maculan's theorem (see theorem \ref{Theorem of Marco Maculan}).

Nevertheless, there is a weaker result. 

\begin{prop}\label{pi^hyb is open under some conditions} (Proposition 4.25 of \cite{RoyActionGroupe})

Let $X$ be an affine scheme over a non-trivially valued field $k,~ G$ be a reductive, algebraic group acting on $X$ with a closed action. 

Let $V \subset X^{hyb}$ be a flowing, $G$-invariant open subset of $X^{hyb} \backslash (\pi^{hyb})^{-1}((X/G)^\beth)$, then 
$\pi^{hyb}(V)$ is open in $(X/G)^{hyb}$.
\end{prop}

\begin{rem}
By proposition 5.4 of \cite{RoyActionGroupe}, in the case where $X = \Rat_d$ and $G = \SL_2$, then $(\pi^{hyb})^{-1}((X/G)^\beth)$ is the set of rational maps with potential good reduction.
\end{rem}

Finally, we can get a result about the continuous motion on $\Rat_d^{hyb}$ and the continuous motion on $M_d^{hyb}$ as in proposition \ref{Julia moves continuously on Rat_d is the same as in M_d}.

\begin{prop}\label{Julia moves continuously on Rat_d is the same as in M_d in hybrid setting}
Let $\C$ be equipped with the hybrid norm and $X$ be an affine $\SL_2$-invariant subscheme of $\Rat_d$. Let $f \in X^{hyb}$ and let $\pi^{hyb} : X^{hyb} \subset \Rat_d^{hyb}\rightarrow (X/\SL_2)^{hyb}\subset \M_d^{hyb}$. Then,
\begin{align*}
\mathrm{There~is~a~flowing~neighborhood~of~}f& \in X^{hyb} \mathrm{~s.~t.~the~Julia~set~moves~continuously}\\
&\iff\\ 
\mathrm{There~is~a~flowing~neighborhood~of~} \pi^{hyb}(f)& \in (X/\SL_2)^{hyb} \mathrm{~s.~t.~the~Julia~set~moves~continuously.} 
\end{align*}
\end{prop}

\begin{rem}
As in proposition \ref{Julia moves continuously on Rat_d is the same as in M_d}, here one only looks at algebraic family of rational maps, but it really is needeed since, the goal will be to study the motion of the Julia set in $\Rat_d^\urcorner$ and $\M_d^\urcorner$ and the hybrid compactification is only defined for a variety over a field $k$. 
\end{rem}

\begin{proof}
$\Longleftarrow$ It is clear.

$\implies$ Let $U$ be a flowing neighborhood of $f \in X^{hyb}$ such that the Julia set moves continuously. We may consider $W$ as in proposition \ref{Construction of SL2 invariant open in hybrid setting} which is $\SL_2$-invariant and flowing. Since the Julia set moves continuously on $W$, for each $f \in W_0$, $f$ has repelling type I points and so $f$ does not have potential good reduction and so $W \subset X^{hyb} \backslash (\pi^{hyb})^{-1}((X/G)^\beth)$. Hence, we may apply the proposition \ref{pi^hyb is open under some conditions} to show that $\pi^{hyb}(W)$ is a neighborhood of $\pi^{hyb}(f)$ and the Julia set moves continuously on $\pi^{hyb}(W)$ by definition of $W$.
\end{proof}

And we also want to compare the continuity on a neighborhood of $f$ in the compactification, $\Rat_d^\urcorner$, of $\Rat_d$ and the continuity on a neighborhood on the image of $f$ in the compactification of $\M_d : \M_d^\urcorner$.

Let us recall the homeomorphism between the quotient of $\Rat_d^\urcorner$ and the compactification og $\M_d^\urcorner$.

\begin{thm}\label{homeomorphism between compactifications} (Théorème 4.23, \cite{RoyActionGroupe})

Let $\C$ be equipped with the hybrid norm and $X$ be an affine $\SL_2$-invariant subscheme of $\Rat_d$. The canonical morphism 
\begin{center}
$\omega : (X^\urcorner \backslash \lbrace f~$have potential good reduction$\rbrace \rightarrow (X/\SL_2)^\urcorner$
\end{center}
is a homemorphism.
\end{thm}

\begin{cor}\label{continuously on compactification of Ratd is the same as on compactification of Md}
Let $\C$ be equipped with the hybrid norm and $X$ be an affine $\SL_2$-invariant subscheme of $\Rat_d$. Let $f\in X^\urcorner$ and $\omega : (X^\urcorner \backslash \mathcal{B})/\SL_2 \rightarrow (X/\SL_2)^\urcorner$ the homeomorphism of theorem \ref{homeomorphism between compactifications}. Then,
\begin{center}
There is a neighborhood of $f \in X^\urcorner$  such that the Julia set moves continuously
\end{center}
\begin{center}
$\iff$
\end{center}
\begin{center}
There is a neighborhood of $\omega(f) \in (X/\SL_2)^\urcorner$  such that the Julia set moves continuously.
\end{center}
\end{cor}

\bigskip

Let us look at the continuity over the hyperbolic connected components of $\M_d(\C)$. Let us recall that a conjugacy class of rational map $[f]$ is called hyperbolic if the orbit of every critical point converges to some attracting periodic cycle. The space $\mathcal{H}$ of hyperbolic rational maps is open in $\M_d(\C)$ and a connected component is called a hyperbolic connected component.

Given this definition, it means that the Julia moves continuously (and holomorphically) on each of the hyperbolic connected components by \cite{MSS}.

First let us recall the notion of blowing-up of multipliers as used by Favre \cite{FavreBlowupMultipliers} and Gong \cite{GongMultiplierScale}.

\begin{defi}
Let $f_i$ be a net of degenerating rational map and $z_i$ repelling points of period $l$. Then, we say that the multipliers blow-up if 
\begin{center}
$\liminf \frac{\log |(f_i^l)'(z_i)|}{\log |\res f_i|} >0$.
\end{center}
\end{defi}

\begin{rem}
In the setting of hybrid spaces, the previous condition is equivalent to asking that if $\varepsilon_i \rightarrow 0$ is such that $f$ is a cluster point of $f_i^{\varepsilon_i}$ in $\Rat_d^{hyb}$ and $\Phi(f)$ is a cluster point of $f_i$ in $\Rat_d^\urcorner$, then
\begin{center}
$\liminf |(f_i^l)'(z_i)|^{\varepsilon_i} >1.$
\end{center}
It comes from the fact that since for all $a_j, b_j, |\frac{a_j^{2d}(f_i)}{\res(f_i)}|^{\varepsilon_i} \rightarrow |\frac{a_j^{2d}(f)}{\res(f)}|$. For $\max(|a_j(f_i)|,| b_j(f_i)|) = 1$, we find $\varepsilon_i \approx \frac{1}{\log|\res f_i|}$.
\end{rem}

We will now look at a condition for a degenerating net $\tilde{f}_i$ to have a Julia set which degenerates continuously. To hope for continuity it is necessary that all of the multipliers of the repelling points blow-up. If it was not the case then, we would have repelling points $z_i$ of $f_i$ of same period $l$ which gives a periodic point $z$ of $f$ of period $l$ and if the multipliers do not blow-up, then $z$ would be an indifferent periodic point and so not in the Julia set which contradicts the continuity of the Julia set.

It is also sufficient if we are in good hyperbolic connected components. One requires that there is no indifferent point in the Julia set of the hyperbolic component, otherwise one cannot have a continuous degeneration of the Julia set since there are no indifferent points in the Julia set in the non-archimedean case.

\begin{thm}\label{InGoodHyperbolicComponentsJuliaSetContinuousIFFMulitpliersBlowUp}
Let $U$ be a hyperbolic connected component of $\M_d(\C)$ such that the Julia sets do not have indifferent periodic point. Let $\tilde{f}_i$ be a degenerating net in $U$.

Then the Julia set of $\tilde{f}_i$ degenerates continuously if and only if all of the multipliers of its repelling points blow-up.
\end{thm}

\begin{proof}
By corollary \ref{continuously on compactification of Ratd is the same as on compactification of Md} one only needs to show that the result stands for any $f_i \in \omega^{-1}(\tilde{f_i})$.

Let $f\in \delta \Rat_d$ be a cluster point of $f_i$. We may assume that $f_i \rightarrow f$. One needs to show that $J(f_i) \rightarrow J(f) \iff$ any multiplier of repelling point blows up.

Let $z_i$ be a repelling point of period $l$ of each $f_i$ and let $C^l$ denote the closed analytic subspace of $\Rat_d^{hyb} \times \p^{1,hyb}$ which characterizes the periodic point of period $l$. We may assume that $z_i$ are all in the same connected component of $C^l$. Let $z \in C^l\cap \pr^{-1}(f)$ a cluster point of the $z_i$. We may assume that $(f_i, z_i) \rightarrow (f,z)$ and so, there exists $\varepsilon_i$ such that $(f_i, z_i)^\epsilon_i \rightarrow (f,z)$ in $(\Rat_d \times \p^1)^{hyb}$.

Then $z$ is a repelling point of period $l$ if and only if $|(f^l)'(z)| = \lim |(f_i^l)'(z_i)|^{\varepsilon_i} >1$ if and only if the multipliers of the repelling points $z_i$ blow-up.

Since the $f_i$ do not have indifferent point in their Julia set, then by theorem \ref{Julia set moves continuously hybrid setting}:
\begin{align*}
J(f_i) \rightarrow J(f) &\iff \mathrm{~all~indifferent~point~of~}f\mathrm{~come~from~non~repelling~points~of~}f_i\\
&\iff \mathrm{~no~indifferent~point~of~}f\mathrm{~comes~from~repelling~points~of~}f_i\\
&\iff \mathrm{~all~multipliers~blow-up.}
\end{align*}
\end{proof}

\begin{rem}
During this section, we worked with net instead of sequences because it gives a more general setting and because the hybrid spaces over $\C$ are not known to be angelic, nevertheless it is possible to only work with sequences and not losing any generality. To do this, instead of considering hybrid spaces as spaces over $\C$ with the hybrid norm, one can consider $\Q[i]$ equipped with its archimedean place and then equip it with a hybrid norm. Hence, theses hybrid spaces are angelic and all the results presented remain true. Instead of seing $\Rat_d$ as a $\C$ variety, one sees it as a $\Q[i]$ variety. The drawback of this is that $\Rat_{d, \Q[i]}^\urcorner$ does not have the same boundary as $\Rat_{d, \C}^\urcorner$. They both have $\Rat_d(\C)$ as a open dense subset but the boundary of $\Rat_{d, \Q[i]}^\urcorner$ is smaller : we have a surjective projection $\delta \Rat_{d, \C}^\urcorner \rightarrow \delta \Rat_{d, \Q[i]}^\urcorner$. Hence some nets share the same limit on $\Rat_{d, \Q[i]}^\urcorner$ but not on $\Rat_{d, \C}^\urcorner$.
\end{rem}

\bibliographystyle{alpha}
\bibliography{biblio_these}
\end{document}